\documentclass[letter]{IEEEtran}

\IEEEoverridecommandlockouts

\usepackage{amsfonts,amssymb,amsmath}
\usepackage{graphicx,graphics,epsfig,color}
\usepackage{multicol}
\usepackage{rgsMacros}

\let\labelindent\relax
\usepackage{enumitem}
\usepackage{balance}
\usepackage{wrapfig}

\definecolor{olivegreen}{rgb}{0.14,0.29,0} 

\newenvironment{proof}{{\it Proof. }}{\hfill $\blacksquare$}

\newtheorem{exe}{Example}
\newtheorem{prob}{Problem}
\newtheorem{corol}{Corollary}
\newtheorem{ass}{Assumption}
\newtheorem{defin}{Definition}
\newtheorem{cla}{Claim}
\newtheorem{rem}{Remark}
\newtheorem{lem}{Lemma}
\newtheorem{prop}{Proposition}
\newtheorem{thm}{Theorem}
\newtheorem{fct}{Fact}
\newenvironment{lemma}{\begin{lem}}{\hfill $\square$ \end{lem}}
\newenvironment{proposition}{\begin{prop}}{\hfill $\square$ \end{prop}}
\newenvironment{corollary}{\begin{corol}}{\hfill $\square$ \end{corol}}
\newenvironment{example}{\begin{exe}}{\hfill $\square$ \end{exe}}
\newenvironment{remark}{\begin{rem}}{\hfill $\bullet$ \end{rem}}

\newenvironment{problem}{\begin{prob}}{\hfill $\bullet$ \end{prob}}

\newenvironment{assumption}{\begin{ass}}{\hfill $\bullet$ \end{ass}}

\newenvironment{theorem}{\begin{thm}}{\hfill $\square$ \end{thm}}
\newenvironment{definition}{\begin{defin}}{\hfill 
$\bullet$ \end{defin}}

\title{\LARGE \bf On the Converse Safety Problem for Differential Inclusions: Solutions, Regularity, and Time-Varying Barrier Functions}

\author{Mohamed Maghenem and Ricardo G. Sanfelice
\thanks{M. Maghenem is with University of Grenoble Alpes, CNRS, Gipsa-lab,  Grenoble INP,  France.  Email: mohamed.maghenem@gipsa-lab.fr.  R. G. Sanfelice is with the Department of Electrical and Computer Engineering, University of California, Santa Cruz. Email:ricardo@ucsc.edu.}
\thanks{This research has been partially supported by the National Science Foundation under Grant no. ECS-1710621, Grant no. CNS-1544396, and Grant no. CNS-2039054, by the Air Force Office of Scientific Research under Grant no. FA9550-19-1-0053, Grant no. FA9550-19-1-0169, and Grant no. FA9550-20-1-0238, and by the Army Research Office under Grant no. W911NF-20-1-0253.}
}

\begin{document}
\maketitle\thispagestyle{empty}\pagestyle{empty}

\begin{abstract}
This paper presents converse theorems for safety in terms of barrier functions for unconstrained continuous-time systems modeled as differential inclusions. Via a counterexample, we show the lack of existence of autonomous and continuous barrier functions certifying safety for a nonlinear system that is not only safe but also has a smooth right-hand side. Guided by converse Lyapunov theorems for (non-asymptotic) stability,  time-varying barrier functions and appropriate infinitesimal conditions are shown to be both necessary as well as sufficient under mild regularity conditions on the right-hand side of the system. More precisely, we propose a general construction of a time-varying barrier function in terms of a marginal function involving the finite-horizon reachable set. Using techniques from set-valued and nonsmooth analysis, we show that such a function guarantees safety when the system is safe. Furthermore, we show that the proposed barrier function construction inherits the regularity properties of the proposed reachable set. In addition, when the system is safe and smooth, we build upon the constructed barrier function to show the existence of a smooth barrier function guaranteeing safety.  Comparisons and relationships to results in the literature are also presented.
\end{abstract}

\section{introduction}    

Beyond stability and convergence, safety is among the most important properties to analyze for a general continuous-time system modeled as the differential inclusion 
\begin{align} \label{eq.1}
\dot{x} \in F(x) \qquad  x \in \mathbb{R}^n.
\end{align} 
Differential inclusions extend the concept of  differential equations by allowing the dynamics to be governed by a set-valued map instead of only a single-valued function \cite{aubin1987differential}. Safety is the property that requires the solutions to \eqref{eq.1} starting from a given set of initial conditions $X_o \subset \mathbb{R}^n$ to never reach a given unsafe region $X_u \subset \mathbb{R}^n$, where, necessarily, $X_o \cap X_u = \emptyset$  \cite{prajna2007framework, wieland2007constructive}.  Safety with respect to 
$(X_o, X_u)$ is verified when a set $K \subset \mathbb{R}^n$, with $X_o \subset K$ and $K \cap X_u = \emptyset$, is \textit{forward pre-invariant}, i.e., the solutions to \eqref{eq.1} starting from $K$ remain in $K$ for all time \cite{prajna2005optimization} --- the prefix ``pre" indicates that solutions may not exist for all $t \in [0,\infty)$, in particular, due to finite escape times. Such a set $K$ is called \textit{inductive invariant} in  \cite{taly2009deductive}. Depending on the considered application, reaching the unsafe set $X_u$ can correspond to the impossibility of applying a predefined feedback law \cite{BELLETER2019123} or, simply colliding with an obstacle \cite{tanner2003stable}. 

\subsection{Background} \label{Sec.intro.backg}

Analogous to Lyapunov theory for stability, the concept of barrier functions is a powerful tool to study safety without computing the solutions to the system. Generally speaking, two main types of barrier functions can be identified in the literature \cite{ames2014control}. The first type of barrier functions consists of a scalar function $B$ defined on the interior of $K$, denoted $\mbox{int}(K)$, with nonnegative values  such that
\begin{align*}  
\lim_{x \rightarrow \partial K} B(x)  = \infty,
\end{align*}
where $\partial K$ is the boundary of $K$. This barrier function certifies safety when the growth condition
\begin{align} \label{eqpot}  
\langle \nabla B (x), \eta \rangle \leq \gamma (B(x)) \qquad
\forall \eta \in F(x), \quad \forall x \in \mbox{int} (K)
\end{align}
is satisfied, where the scalar function $\gamma$ is such that condition \eqref{eqpot} implies that the map $t \mapsto B(\phi(t,x_o))$ does not become unbounded in finite time for every solution $\phi$ to \eqref{eq.1} starting from 
$x_o \in \mbox{int}(K)$ --- each such solution is denoted $t \mapsto \phi(t,x_o)$. Hence, the solution $\phi$ remains in $\mbox{int}(K)$ for all time. 
This type of barrier functions, often named \textit{potential} functions, has been used in constrained optimization \cite{WILLS20041415}, multiagent systems \cite{tanner2003stable}, and constrained nonlinear control design \cite{tee2009barrier}. 

The second type of barrier functions is given by a scalar function $B$ with a prescribed sign on the initial set $X_o$ and with the opposite sign on the unsafe set $X_u$. Without loss of generality, we can assume that $B$ and $(X_o,X_u)$ satisfy
\begin{align} \label{eq.2}
\begin{matrix} 
B(x) > 0 & \forall x \in X_u \\ 
B(x) \leq 0 & \forall x \in X_o. 
\end{matrix} 
\end{align} 
In this case, safety is guaranteed when 
the zero-sublevel set 
\begin{align} \label{eqaddd} 
K := \left\{ x \in \mathbb{R}^n : B(x) \leq 0 \right\}
\end{align} 
is forward pre-invariant. The first characterization of forward pre-invariance dates back to the work of Nagumo in \cite{nagumo1942lage}, where tangent-cone-based conditions are proposed; see Section \ref{sec.apNagumo} for more details. Note that the computation of a tangent cone to a general set is not always a trivial task. Fortunately, when the set $K$ satisfies \eqref{eqaddd}, it is possible to formulate sufficient conditions for forward pre-invariance using only the barrier candidate $B$ and 
the right-hand side of \eqref{eq.1}, $F$. Such sufficient conditions are usually expressed in terms of an inequality constraining the variation of $B$ along the solutions to the system \eqref{eq.1}. In \cite[Proposition 2]{prajna2007framework}, the condition 
\begin{align} 
\langle \nabla B (x), \eta \rangle \leq 0 \qquad \forall \eta \in F(x), \quad \forall x \in \mathbb{R}^n \label{eq.2c}
\end{align}
is used. Condition \eqref{eq.2c} has been relaxed in the literature in different ways. According to our previous work in \cite{draftautomatica}, the inequality in \eqref{eq.2c} does not need to hold on the entire 
$\mathbb{R}^n$ to guarantee forward pre-invariance. It is enough to guarantee that
\begin{align} 
\langle \nabla B (x), \eta \rangle \leq 0 \qquad 
\forall \eta \in F(x), \quad \forall x \in U(K) \backslash K, \label{eq.2c1}
\end{align}
where $U(K)$ is any open neighborhood around the (closed) set $K$. Furthermore, according to \cite[Theorem 1]{prajna2004safety}, when $F$ is locally Lipschitz and $\nabla B(x) \neq 0$ for all $x$ in the boundary of $K$ denoted  $\partial K$,  the inequality in \eqref{eq.2c} can be relaxed to hold only on the boundary of $K$; namely, it is enough to assume 
\begin{align} 
\langle \nabla B (x), \eta \rangle \leq 0 \qquad  
\forall \eta \in F(x), \quad \forall x \in \partial K. \label{eq.2c2}
\end{align}

The non-positiveness required in \eqref{eq.2c} and \eqref{eq.2c1} can be relaxed using uniqueness functions, or, minimal functions; see Section \ref{sec.4-} for more details. It is important to note that conditions \eqref{eq.2c}, \eqref{eq.2c1}, and \eqref{eq.2c2} require continuous differentiability of the barrier function candidate $B$. Similar conditions can be formulated when $B$ is only locally Lipschitz or only lower semicontinuous, using  appropriate tools; see \cite{draftautomatica}. In the most general case where $B$ is not necessarily smooth, the aforementioned conditions can be replaced by the following solution-dependent monotonicity property:
\begin{enumerate} [label={($\star$)},leftmargin=*]
\item \label{item:star}  Along each solution $\phi$ to \eqref{eq.1} starting from $x_o \in U(K) \backslash \mbox{int}(K)$ and such that 
$\phi([0,T],  x_o) \subset U(K) \backslash \mbox{int}(K)$, for some $T>0$ , the map $t \mapsto B(\phi(t,x_o))$ is nonincreasing on $[0,T]$. 
\hfill $\bullet$
\end{enumerate}
The second type of barrier functions in \eqref{eq.2} has been applied to multi-robots collision avoidance in \cite{glotfelter2017nonsmooth, 8625554}, adaptive cruise control in \cite{xu2018correctness}, and bipedal walking in \cite{nguyen2015safety}. 

Finally, a notion equivalent to safety, named \textit{conditional invariance}, is studied and characterized in \cite{ladde1974, ladde1972analysis, lakshmikantham1969differential, kayande1966conditionally} using Lyapunov-like conditions. Roughly speaking, a set $X_s \subset \mathbb{R}^n$ is conditionally invariant with respect to a set $X_o \subset X_s$ if the solutions starting from $X_o$ never leave the set $X_s$. Connections between Lyapunov-like conditions guaranteeing conditional invariance and the more recent conditions using barrier functions are discussed in Section \ref{appen1}. 

\subsection{Motivation} \label{secIntroMotiv}

Many existing tools to certify safety for control systems are based on the search of a controller and the corresponding barrier function that certifies safety for the resulting closed-loop system  \cite{DAI201762, 10.1007/978-3-642-39799-8_17, robey2021learning}.  By solving the converse safety problem, in this case,  one can be assured that a barrier function exists  when the control system can be rendered safe.   Generally speaking,  given a safe system \eqref{eq.1} with respect to  $(X_o, X_u)$, the \textit{converse safety problem} pertains to showing the existence of a barrier function candidate $B : \mathbb{R}^n \rightarrow \mathbb{R}$ satisfying \eqref{eq.2} and verifying conditions guaranteeing safety, such as those in \eqref{eq.2c}, \eqref{eq.2c1}, \eqref{eq.2c2}, and \ref{item:star}.  To the best of our knowledge, \cite{prajna2005necessity}, \cite{wisniewski2016converse}, and \cite{ratschan2018converse} are the only existing works treating the converse safety problem via barrier functions. We review these results next. 
 
The converse safety result proposed in \cite{prajna2005necessity} applies when $F$ is single valued and continuously differentiable. Furthermore, it assumes that there exists a continuously differentiable function $V : \mathbb{R}^n \rightarrow \mathbb{R}$ that is strictly decreasing along the solutions to \eqref{eq.1}; namely, $V$ and $F$ satisfy 
\begin{align} \label{eq.2cd} 
\langle \nabla V(x), F(x) \rangle < 0 \qquad  \forall x \in \mathbb{R}^n. 
\end{align} 
Under these conditions, safety with respect to 
$(X_o,X_u)$ is shown to imply the existence of a continuously differentiable barrier function candidate $B$ satisfying \eqref{eq.2c}. Note that this result does not apply when system \eqref{eq.1} admits a limit cycle. Indeed, for systems with limit cycles, it is not possible to find a function $V$ such that \eqref{eq.2cd} holds; see Example \ref{exp}. 

In \cite{wisniewski2016converse}, a geometric point of view is adopted using \textit{Morse-Smale} theory when  system \eqref{eq.1} is defined on a smooth and compact manifold. The right-hand side $F$ is assumed to be single valued and smooth. Also, the sets $X_o$ and $X_u$ are assumed to be compact and disjoint. In the study in \cite{wisniewski2016converse}, a robust safety notion (see Definition \ref{defrobsaf}) is introduced, for which necessary and sufficient conditions using barrier functions are proposed. Furthermore, in the converse safety result in \cite{wisniewski2016converse}, the strictly decreasing function $V$ assumed to exist in \cite{prajna2005necessity} is replaced by the existence of a \textit{Meyer} function; see \cite[Definitions 7 and 8]{wisniewski2016converse} for more details. 

Finally, in \cite{ratschan2018converse}, a converse robust safety result that does not assume existence of $V : \mathbb{R}^n \rightarrow \mathbb{R}$ such that \eqref{eq.2cd} holds nor the existence of a Meyer function is established when $F$ is smooth and single valued. According to the latter reference, system \eqref{eq.1} is robustly safe with respect $(X_o,X_u)$ if, for some $\epsilon>0$, the perturbed system 
\begin{align} \label{eq.1Inf}
\dot{x} \in F(x) + \epsilon \mathbb{B} \qquad  x \in \mathbb{R}^n,
\end{align} 
where $\mathbb{B} \subset \mathbb{R}^n$ is the closed unit ball centered at the origin, is safe with respect $(X_o,X_u)$. It is shown in \cite{ratschan2018converse} that when additionally the closures of the sets $X_o$ and $X_u$ are disjoint, and the set 
$\mathbb{R}^n \backslash X_u$ is bounded, robust safety of system \eqref{eq.1} with respect to $(X_o,X_u)$ is equivalent to the existence of a barrier function candidate satisfying \eqref{eq.2} and such that 
\begin{align*} 
\langle \nabla B (x), F(x) \rangle < 0 \qquad 
\forall x \in \partial K.
\end{align*}

To the best of our knowledge, providing necessary and sufficient conditions for safety, or robust safety, without restricting the class of systems \eqref{eq.1}, are not available in the literature. 
Furthermore, as we show in this paper, safe systems may not admit a barrier function with the properties assumed in the literature. In fact, Example \ref{expcount} presents a system as in \eqref{eq.1} that is safe with respect to $(X_o,X_u) \subset \mathbb{R}^n \times \mathbb{R}^n$, where $F$ is single valued and smooth, but does not admit a barrier function candidate $B : \mathbb{R}^n \rightarrow \mathbb{R}$, function of $x$ only, that is continuous and satisfies any of the sufficient conditions for safety in \eqref{eq.2c}, \eqref{eq.2c1}, \eqref{eq.2c2}, and \ref{item:star}. This fact motivates the new class of barrier functions introduced in this paper.

\subsection{Contributions}

This paper makes the following contributions: 

\begin{enumerate}
\item  We formulate a safety problem in terms of time-varying barrier functions, that are not necessarily smooth, and propose necessary and sufficient conditions for safety without assuming existence of $V : \mathbb{R}^n \rightarrow \mathbb{R}$ such that \eqref{eq.2cd} holds, the existence of a Meyer function, or boundedness of the set $\mathbb{R}^n \backslash X_u$. Allowing for 
nonsmooth barrier functions is justified by the lack of existence of smooth scalar functions satisfying \eqref{eq.2} for some scenarios of sets $(X_o,X_u)$ as shown in Example \ref{exp3}. Furthermore, time-varying barrier functions are motivated by the existing converse Lyapunov theorems for stability, where time-varying Lyapunov functions are constructed for systems with a stable origin \cite{Persidskiipap, kurzweil1955, kurzweil1957, temple1965stability, hahn1967stability}. 

\item In Section \ref{sec.3a}, inspired by the converse Lyapunov stability theorem in \cite{Persidskiipap}, given initial and unsafe sets $(X_o,X_u)$, we construct a time-varying barrier function as a marginal function of an appropriately defined  reachable set over a given finite window of time, along the solutions to \eqref{eq.1}, and starting from a given initial condition. We show that such a barrier function guarantees safety provided that \eqref{eq.1} is safe with respect to $(X_o,X_u)$. 

\item Furthermore, we show that this barrier function inherits the regularity properties of the proposed reachable set when this one is viewed as a set-valued map \cite{aubin2012differential}. As a result, when $F$ satisfies mild regularity conditions, we show that safety of \eqref{eq.1} with respect to $(X_o,X_u)$ is equivalent to the existence of a lower semicontinuous time-varying barrier function; see Theorem \ref{thm2}.    

\item In Section \ref{sec.3b}, when in addition $F$ is locally Lipschitz, we establish Lipschitz continuity of the proposed reachability map using Filippov Theorem \cite[Theorem 5.3.1]{Aubin:1991:VT:120830}. As a result, using the dependence of the constructed barrier function on the reachability map, we conclude that safety is equivalent to the existence of a locally Lipschitz time-varying barrier function; see Theorem \ref{thm3}.    

\item In Section \ref{sec.3c}, inspired by the converse Lyapunov stability theorem in \cite{kurzweil1955}, we build upon the barrier function constructed in Section \ref{sec.3b} to conclude the existence of a barrier function that is continuously differentiable provided that $F$ is single valued and continuously differentiable; see Theorem \ref{prop12}. As observed in \cite{kurzweil1957}, Lyapunov stability of the origin is equivalent to conditional invariance with respect to a sequence of compact sets 
$\{(X_{oi},X_{si})\}^{\infty}_{i= 0}$ that converges to the origin. However, extending the converse stability result in \cite{kurzweil1955} to the context of safety is not straightforward and offers many technical challenges. Those challenges are due to the fact that the sets $X_o$ and 
$\mathbb{R}^n \backslash X_u$ are not necessarily bounded, $X_o$ is not necessarily forward pre-invariant, and the solutions to the system are not necessarily bounded.
\end{enumerate}

Preliminary version of this work is in \cite{magh2018conditional1}, where only differential equations are considered and the proofs are omitted. 
Furthermore, the current paper includes more examples and a more detailed comparison to the existing literature.

The remainder of the paper is organized as follows. Preliminary notions are in Section \ref{sec.1}. The converse safety problem using time-varying barrier functions is formulated in Section \ref{sec.2}. The main results are in Section \ref{sec.3}. A comparison to existing literature is in Section \ref{sec.4-}. Finally, conclusion and future work are in Section \ref{sec.4}.

\textbf{Notation.}
Let $\mathbb{R}_{\geq 0} := [0, \infty)$,  $\mathbb{N} := \left\{0,1,\ldots \right\}$, and $\mathbb{N}^*:= \left\{1, 2,  \ldots, \infty \right\}$. 
For $x$ and $y \in \mathbb{R}^n$, $x^{\top}$ denotes the transpose of $x$, $|x|$ the Euclidean norm of $x$, and $\langle x, y \rangle := x^\top y$ denotes the scalar product between $x$ and $y$. 
For a set $K \subset \mathbb{R}^n$, we use $\mbox{cl}(K)$ to denote its closure and $|x|_K := \inf_{y \in K} |x-y|$ to define the distance between $x$ and the set $K$. For $O \subset \mathbb{R}^n$, $K \backslash O$ denotes the subset of elements of $K$ that are not in $O$. By $\mathbb{B}$, we denote the closed unite ball centered at the origin. By $F: \mathbb{R}^n \rightrightarrows \mathbb{R}^n $, we denote a set-valued map associating each element $x \in \mathbb{R}^n$ into a subset $F(x) \subset \mathbb{R}^n$. For a set-valued map $F : \mathbb{R}^n \rightrightarrows \mathbb{R}^m$, $\dom F$ denotes the domain of definition of $F$ and $F^{-1}(x)$ denotes the reciprocal image of $F$ evaluated at $x$. For a continuously differentiable function $B: \mathbb{R}^n \rightarrow \mathbb{R}$, $\nabla B(x)$ denotes the gradient of $B$ evaluated at $x$. 
Finally,  by $\mathcal{C}^k(K)$, with $k \in \mathbb{N}$, we denote the class of $k-$times differentiable functions on $K$ where the $k-$th derivative is continuous on $K$ (when $K = \mathbb{R}^n$, we only write $\mathcal{C}^k$). 

\section{Preliminaries} \label{sec.1}

\subsection{Set-Valued and Single-Valued Maps} \label{sec.1a}

We start this section by recalling the following continuity notions for set-valued and single-valued maps. 

\begin{definition} [Semicontinuous set-valued maps] \label{deflusc}
Consider a set-valued map $F: K \rightrightarrows \mathbb{R}^n$, where $K \subset \mathbb{R}^m$.
\begin{itemize}
\item The map $F$ is said to be \textit{outer semicontinuous} at $x \in K$ if, for every sequence $\left\{x_i\right\}^{\infty}_{i=0} \subset K$ and for every sequence  
$\left\{ y_i \right\}^{\infty}_{i=0} \subset \mathbb{R}^n$ with 
$\lim_{i \rightarrow \infty} x_i = x$, $\lim_{i \rightarrow \infty} y_i = y \in \mathbb{R}^n$, and $y_i \in F(x_i)$ for all $i \in \mathbb{N}$, we have $y \in F(x)$; 
see \cite[Definition 5.9]{goebel2012hybrid}. 
\item The map $F$ is said to be \textit{lower semicontinuous}  (or, equivalently, \textit{inner semicontinuous}) at $x \in K$ if for each 
$\epsilon > 0$ and $y_x \in F(x)$, there exists $U(x)$ satisfying the following property: for each $z \in U(x) \cap K$, there exists $y_z \in F(z)$ such that $|y_z - y_x| \leq \epsilon$; see  \cite[Proposition 2.1]{michael1956continuous}.  
\item The map $F$ is said to be \textit{upper semicontinuous} at $x \in K$ if, for each 
$\epsilon > 0$, there exists $U(x)$ such that for each $y \in U(x) \cap K$, $F(y) \subset F(x) + \epsilon \mathbb{B}$; see \cite[Definition 1.4.1]{aubin2009set}.
\item The map $F$ is said to be \textit{continuous} at $x \in K$ if it is both upper and lower semicontinuous at $x$.
\end{itemize}
Furthermore, the map $F$ is said to be upper, lower, outer semicontinuous, or continuous if, respectively, it is upper, lower, outer semicontinuous, or continuous for all $x \in K$.
\end{definition}

\begin{definition}[Semicontinuous single-valued maps] \label{defluscbis}
Consider a scalar function $B: K \rightarrow \mathbb{R}$, where $K \subset \mathbb{R}^m$.
\begin{itemize}
\item The scalar function $B$ is said to be \textit{lower semicontinuous} at $x \in K$ if, for every sequence $\left\{ x_i \right\}_{i=0}^{\infty} \subset K$ such that $\lim_{i \rightarrow \infty} x_i = x$, we have $\liminf_{i \rightarrow \infty} B(x_i) \geq B(x)$. 
\item The scalar function $B$ is said to be \textit{upper semicontinuous} at $x \in K$ if, for every sequence $\left\{ x_i \right\}_{i=0}^{\infty} \subset K$ such that $\lim_{i \rightarrow \infty} x_i = x$, we have $\limsup_{i \rightarrow \infty} B(x_i) \leq B(x)$.  
\item The scalar function $B$ is said to be \textit{continuous} at $x \in K$ if it is both upper and lower semicontinuous at $x$. 
\end{itemize}
Furthermore, $B$ is said to be upper semicontinuous, lower semicontinuous, or continuous if, respectively, it is upper semicontinuous, lower semicontinuous, or continuous for all $x \in K$.
\end{definition} 
  
\begin{definition}[Locally bounded set-valued maps]
A set-valued map $F: K \rightrightarrows \mathbb{R}^n$, with $K \subset \mathbb{R}^m$, is said to be \textit{locally bounded} if, for any $x \in K$, there exist $U(x)$ and $\beta > 0$ such that $|\zeta| \leq \beta$ for all $\zeta \in F(y)$ and for all $y \in U(x) \cap K$.  
\end{definition}

\begin{definition} [Locally Lipschitz set-valued maps] \label{deflip}
The set-valued map $F : K \rightrightarrows \mathbb{R}^n$, with $K \subset \mathbb{R}^m$, is said to be \textit{locally Lipschitz} if, for each nonempty set $K_o \subset K$, there exists $k>0$ such that, 
for all $(x_1,x_2) \in K_o \times K_o$, 
\begin{align} \label{eq.lipset}
F(x_1) \subset F(x_2) + k |x_1-x_2| \mathbb{B},  
\end{align}
or, equivalently,
\begin{align} \label{eq.lipsetb}
d_H(F(x_2), F(x_1)) \leq  k |x_2-x_1|,  
\end{align}
where $d_H(X_1, X_2)$ is the Hausdorff distance between the sets $X_1 \subset \mathbb{R}^m$ and 
$X_2 \subset \mathbb{R}^m$; namely, 
\begin{align} \label{eq.Haus}
d_H(X_1, X_2) := \max \left\{ \sup_{x \in X_1} |x|_{X_2}, \sup_{x \in X_2} |x|_{X_1} \right\}.
\end{align}
\end{definition}

\begin{definition} [Locally Lipschitz functions] \label{deflipfun}
A function $F : K \rightrightarrows \mathbb{R}^n$, with $K \subset \mathbb{R}^m$, is said to be \textit{locally Lipschitz} if, for each nonempty set $K_o \subset K$, there exists $k>0$ such that, 
for all $(x_1,x_2) \in K_o \times K_o$, 
\begin{align} \label{eq.lipfun}
|F(x_1) - F(x_2)| \leq k |x_1-x_2|.  
\end{align}
\end{definition}

\begin{definition}[Epigraph of functions]
Given a scalar function $B : \mathbb{R}^n \rightarrow \mathbb{R}$, its epigraph is given by 
\begin{align} 
\epi B := & \left\{ (x, r) \in \mathbb{R}^n \times \mathbb{R} : r \geq B(x) \right\}. \label{eq:epi}
\end{align}
\end{definition}

\begin{definition}[Regular sets and functions] \label{def.reg}
A set $K \subset \mathbb{R}^n$ is said to be regular if $T_K(x) = C_K(x)$ for all $x \in K$, 
where $T_K$ and $C_K$ are the \textit{contingent} and the \textit{Clarke tangent} cones of $K$ at $x$, respectively, and  given by
\begin{align} 
T_K(x) & := \left\{ v \in \mathbb{R}^n: \liminf_{h \rightarrow 0^+} \frac{|x + h v|_K}{h} = 0 \right\}. \label{eq.toncon} 
\\
C_K(x) & := \left\{ v \in \mathbb{R}^n: \limsup_{y \rightarrow x, h \rightarrow 0^+} \frac{|y+ h v|_K}{h} = 0 \right\}. \label{eq.tonconc} 
\end{align}
Furthermore, a locally Lipschitz function $B : \mathbb{R}^n \rightarrow \mathbb{R}$ is regular if $\epi B$ is regular. 
\end{definition}

\begin{remark}
The definition of regular functions used in this paper is equivalent to the definition used in \cite{clarke2008nonsmooth}; 
see Proposition 7.3 therein. 
\end{remark}

\subsection{Proximal Subdifferential and Clarke Generalized Gradient} \label{sec.1b}

In this section, we recall from  \cite{clarke2008nonsmooth} the tools to certify safety using nonsmooth barrier function candidates.

\begin{definition} [Proximal normal cone] \label{def-prox}
Given a set $S \subset \mathbb{R}^n$, the proximal normal cone $N^P_S$ associated with  $S$ evaluated at $x \in \mbox{cl}(S)$ is given by
\begin{align} \label{pro-cone}
\hspace{-0.2cm} N_S^P(x) := \left\{ \zeta \in \mathbb{R}^n : \exists r > 0~\mbox{so that}~|x+r\zeta|_{S} = r 
|\zeta| \right\}.
\end{align}
\end{definition}

\begin{definition}[Proximal subdifferential]  \label{defps}
The proximal subdifferential of a lower semicontinuous function $B: \mathbb{R}^n \rightarrow \mathbb{R}$ is the set-valued map $ \partial_P B : \mathbb{R}^n \rightrightarrows \mathbb{R}^n $ such that, for all $x \in \mathbb{R}^n$,
\begin{align} \label{eq.subgrad}
\hspace{-0.2cm} \partial_P B(x) := \left\{ \zeta \in \mathbb{R}^n : [\zeta^\top~-1]^\top \in N^P_{\epi B} (x, B(x)) \right\}.
\end{align}
Moreover, each vector $\zeta \in \partial_P B(x)$ is said to be a \textit{proximal subgradient} of $B$ at $x$. 
\end{definition}

\begin{remark}
Using \cite[Theorem 2.5]{clarke2008nonsmooth}, we conclude that 
\begin{align} \label{eq.subgradbis}
\partial_P B(x) = & \left\{ \zeta \in \mathbb{R}^n : \exists U(x),~\exists \epsilon > 0 : \forall y \in U(x) \right. \nonumber \\ & \left.
B(y) \geq B(x) + \langle \zeta , y-x \rangle - \epsilon |y-x|^2 
\right\}.
\end{align}
Furthermore, when $B \in \mathcal{C}^2$, we conclude that $\partial_P B(x) = \left\{ \nabla B(x) \right\}$. Moreover, the latter equality holds also when $B$ is only $\mathcal{C}^1$ provided that $\partial_P B(x) \neq \emptyset$. 
\end{remark} 

\begin{definition} [Clarke generalized gradient] \label{defgen}
Let $B : \mathbb{R}^n \rightarrow \mathbb{R}$ be locally Lipschitz. Let $\Omega$ be any subset of zero measure in $\mathbb{R}^n$,
and let $\Omega_B$ be the set of points in $\mathbb{R}^n$ at which $B$ fails to be differentiable. The Clarke generalized gradient at $x$ is defined as
\begin{align} \label{eq.gg}
\partial_C B(x) := \co \left\{ \lim_{i \rightarrow \infty} \nabla B(x_i) : x_i \rightarrow x,~x_i \notin \Omega_B,~x_i \notin \Omega \right\}.
\end{align}
\end{definition} 

\begin{remark}
Definition \ref{defgen} is equivalent to the original definition of the Clarke generalized gradient in \cite{clarke2008nonsmooth}; see Theorem 8.1 therein.
\end{remark}

\subsection{Safety and Set-Invariance in Differential Inclusions} \label{sec.1d}

First, we recall the concept of solution to \eqref{eq.1}.

\begin{definition} [Concept of solution] \label{def.CS}
A function $\phi : \dom \phi \rightarrow \mathbb{R}^n$, where $\dom \phi$ is of the form $[0,T]$ or $[0,T)$ for some $T \in \mathbb{R}_{\geq 0} \cup \{+ \infty\}$, is a solution to \eqref{eq.1} starting from $x_o \in \mathbb{R}^n$ if $t \mapsto \phi(t,x_o)$ is locally absolutely continuous and satisfies \eqref{eq.1} for almost all $t \in \dom \phi$.
\end{definition}
 A solution $\phi$ starting from $x_o \in \mathbb{R}^n$ is forward complete if $\dom \phi$ is unbounded, and it is maximal if there is no solution $\psi$ starting from $x_o$ such that $\psi(t,x_o) = \phi(t,x_o)$ for all $t \in \dom \phi$ and $\dom \phi$ is a proper subset of $\dom \psi$. 
 Finally, the system \eqref{eq.1} is said to be forward complete if each of its maximal solutions is forward complete. 

Next, we consider a set $X_u \subset \mathbb{R}^n$ denoting the unsafe region of the state space, a set $X_o \subset \mathbb{R}^n$ denoting the set of initial conditions -- namely, the region that the solutions start from -- and a set $X_s$ denoting the safe set. Without loss of generality, we assume that $X_o \cap X_u = \emptyset$, $X_o \subset X_s$, and $X_s \cap X_u = \emptyset$.

\begin{definition}[Safety] \label{def-2} 
System \eqref{eq.1} is said to be safe with respect to 
$(X_o, X_u)$ if, for each solution $\phi$ to \eqref{eq.1} starting from 
$x_o \in X_o$, we have $\phi(t, x_o) \in \mathbb{R}^n \backslash X_u$ for all $t \in \dom \phi$.
\end{definition}  
\begin{definition} [Conditional invariance \cite{ladde1974}] \label{def-1} 
A set $X_s \subset \mathbb{R}^n$ is conditionally invariant with respect to a set $X_o \subset X_s$ for system \eqref{eq.1} if, for each solution $\phi$ starting from $x_o \in X_o$, we have
$\phi(t, x_o) \in X_s$ for all $t \in \dom \phi$.  
\end{definition}
\begin{definition} [Forward pre-invariance] \label{def-3} 
A set $X_s \subset \mathbb{R}^n$ is forward pre-invariant for 
\eqref{eq.1} if, for each solution $\phi$ to \eqref{eq.1} starting from $x_o \in X_s$, we have $\phi(t, x_o) \in X_s$ for all $t \in \dom \phi$.  
\end{definition}

The safety and the conditional invariance notions are related as follows: system \eqref{eq.1} is safe with respect to $(X_o, X_u)$ if and only if the set $X_s := \mathbb{R}^n \backslash X_u$ is conditionally invariant with respect to $X_o$ for \eqref{eq.1}. Safety generalizes the forward pre-invariance notion: 
forward pre-invariance of a set $X_s \subset \mathbb{R}^n$ is equivalent to safety with respect 
$(X_s , \mathbb{R}^n \backslash X_s)$. Note that, the prefix ``pre'' in forward pre-invariance is used to accommodate maximal solutions that are not complete. For example, if a solution $\phi$ to \eqref{eq.1} starts from $x_o \in X_s$ and has a finite-time escape while remaining in $X_s$, then such a solution may still satisfy $\phi(t,x_o) \in X_s$ for all $t\in \dom \phi$, but with $\dom \phi$ bounded and open to the right.
 
\section{The Converse-Safety Problem Formulation} \label{sec.2}

Generally speaking, converse safety theorems identify classes of dynamical systems for which safety is equivalent to the existence of a smooth barrier function satisfying \eqref{eq.2} plus a sufficient condition for safety.  According to the following (counter) example, for the system in \eqref{eq.1} that is safe with respect to $(X_o,X_u)$,  it is not always possible to find a barrier function candidate $B : \mathbb{R}^n \rightarrow \mathbb{R}$, function of $x$ only,  that is continuous and such that both  \eqref{eq.2} and \ref{item:star} hold. 

\begin{example} \label{expcount}
Consider the system in \eqref{eq.1} with $x \in \mathbb{R}^2$,  
\begin{equation} 
\label{cont.exp}
\begin{aligned} 
F(x) := \left\{
\begin{matrix}
\begin{bmatrix}
 -x_2 + r x_1 \sin(1/r)^2  \\ 
 x_1 + r x_2 \sin (1/r)^2
 \end{bmatrix} & \text{if} ~ x \neq 0
\\ & \\
0 & \text{otherwise}, 
\end{matrix} 
\right. 
\end{aligned}
\end{equation}
and $r := |x|$.   The system is safe with respect to the sets
\begin{align} \label{eqSetsCounterExp} 
X_o := \left\{ 0 \right\},~~ X_u := \mathbb{R}^2 \backslash X_o. 
\end{align}
Indeed, the safety property, in this case, is equivalent to forward invariance of the origin (which coincides with $X_o$). Forward invariance of the origin holds since the origin is an equilibrium point for system \eqref{cont.exp}. However, we show below that it is not possible to find a barrier candidate $B$, function only of $x$, that is continuous, nonincreasing along the solutions to the system, and at the same time having a value at the origin that is strictly smaller than all the values elsewhere as \eqref{eq.2} requires. 

In polar coordinates, system \eqref{cont.exp} can be rewritten as
\begin{align} \label{cont.exp1}
\dot{r} = (r^2/2) \sin (1/r)^2, \qquad \dot{\theta} = 1.
\end{align}
Furthermore, from \eqref{cont.exp1}, it follows that the origin is surrounded by (countably) infinitely many limit cycles centered at the origin, denoted by $Q_i$, $i \in \mathbb{N}$. Moreover, the radius of the limit cycles monotonically converges to zero as $i \rightarrow \infty$ and the trajectories starting from the interior of the annulus formed by each two circles $Q_{i+1}$ and $Q_i$ are spirals that leave $Q_{i+1}$ and approach $Q_i$. Figure \ref{Fig:1} depicts such limit cycles as well as solutions starting from different initial conditions.
\begin{figure} \label{Fig:1}
\begin{center}
\includegraphics[width= 1 \columnwidth]{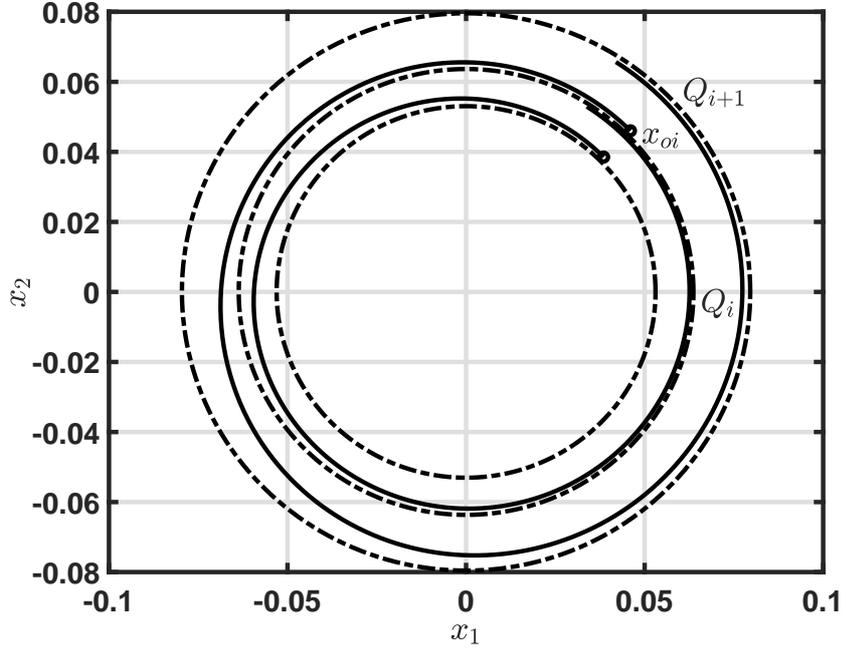}
\caption{Solutions to system \eqref{cont.exp} starting from different initial conditions.}
\end{center}
\end{figure}

Now, assume the existence of a continuous function $B$ that is nonincreasing along the solutions to \eqref{cont.exp} and positive definite. Furthermore, for a sequence of points 
$\left\{ x_i \right\}^{\infty}_{i=0}$ with $x_i \in Q_i$, the sequence $\left\{ B(x_i) \right\}^{\infty}_{i=0}$ converges to zero, and is strictly positive. Hence, there exists a strictly positive and monotonically decreasing subsequence $\left\{B(x_{i_k})\right\}^{\infty}_{k=0}$ that also converges to zero. As a result, there exist $(l_1,l_2) \in \mathbb{N} \times \mathbb{N}$ and $\epsilon > 0$ such that $ B(x_{l_1}) - B(x_{l_2}) = \epsilon $. We assume, further and without loss of generality, that $l_{2} - l_{1} = 2$ (the same reasoning is valid if $l_{2} - l_{1} > 2$). Next, using the continuity assumption on $B$ and the properties of solutions to \eqref{cont.exp}, it follows that for any $\epsilon_1>0$ we can find $T> 0$ and two initial conditions $x_o$ and $x_{o1}$ in the interior of the annulus formed by $Q_{l_2}$ and $Q_{l_{2}-1}$ and, respectively, in the interior of the torus formed by $Q_{l_2-1}$ and $Q_{l_1}$ such that 
\begin{align*}
\max \left\{ | B(x_o) - B(x_{l_2})|, \right. & \left.|B(x_{o1}) - B(\phi(T,x_o))|, \right. \\ & \left. |B(x_{l_1}) - B(\phi_1(T,x_{o1}))| \right\} \leq \epsilon_1,
\end{align*} 
where $\phi$ and $\phi_1$ are the solutions to \eqref{cont.exp} starting from $x_o$ and $x_{o1}$, respectively. 
Now, having 
\begin{align*}
\epsilon = & B(x_{l_1}) - B(x_{l_2}) = B(x_{l_1}) - B(\phi_1(T,x_{o1})) + \\ & B(\phi_1(T,x_{o1})) - B(x_{o1}) + B(x_{o1}) - B(\phi(T,x_o)) + \\ & B(\phi(T,x_o)) - B(x_o) + B(x_o) - B(x_{l_2})      
\end{align*}
and using the fact that $B$ does not increase along the solutions to system \eqref{cont.exp}, we obtain
 \begin{align*}
 \epsilon = & B(x_{l_1}) - B(x_{l_2})  \leq |B(x_{l_1}) - B(\phi_1(T,x_{o1}))| + \\ & |B(x_{o1}) - B(\phi(T,x_o))| + | B(x_o) - B(x_{l_2})| \leq 3 \epsilon_1.
\end{align*}
The latter fact yields to a contradiction since $\epsilon$ is fixed and $\epsilon_1$ can be made as small as possible, that is, for $\epsilon_1 = \epsilon/4$, we obtain $\epsilon \leq 3 \epsilon/4$ which is a contradiction. Hence, though it is safe, an autonomous barrier function does not exist.   
\end{example}

This example  is inspired from \cite[Page 82]{hahn1967stability} and \cite[Page 46]{temple1965stability},  where the existence of Lyapunov functions for (non-asymptotically) stable systems is analyzed. 

To handle the lack of existence of smooth barrier functions for safe systems, we introduce the following  time-varying barrier function candidate notion. 

\begin{definition}[Time-varying barrier function candidate] \label{defTVBf}
A scalar function $B : \mathbb{R}_{\geq 0} \times \mathbb{R}^n \rightarrow \mathbb{R}$ is a time-varying barrier function candidate for safety with respect to $(X_o,X_u)$ if 
\begin{align} 
 B(t,x) & > 0 \qquad \forall (t,x) \in \mathbb{R}_{\geq 0} \times  X_u, \label{eq.2a+} \\ 
 B(t,x) & \leq 0 \qquad \forall (t,x) \in \mathbb{R}_{\geq 0} \times  X_o. \label{eq.2b+} 
\end{align}
\end{definition}

Using time-varying barrier functions, we will be able to address the following converse safety problem.

\begin{problem}[Converse safety problem] \label{prob1}
Given sets $(X_o,X_u) \subset \mathbb{R}^{n} \times \mathbb{R}^n$, with $X_o \cap X_u = \emptyset$, show that the system in \eqref{eq.1} is safe with respect to $(X_o,X_u)$ if and only if there exists a 
time-varying barrier function candidate $B : \mathbb{R}_{\geq 0} \times \mathbb{R}^n \rightarrow \mathbb{R}$, with the best possible regularity\footnote{By ``best regularity'',  we mean the strongest smoothness property.},  such that
\begin{enumerate} [label={($\star \star$)},leftmargin=*]
\item \label{item:starstar} 
Along each solution $\phi$ to \eqref{eq.1} starting from $x_o \in U(K) \backslash \mbox{int}(K)$ and remaining in $U(K) \backslash \mbox{int}(K)$, the map $t \mapsto B(t,\phi(t,x_o))$ is nonincreasing, where 
\begin{align} \label{eqadddbis} 
K := \left\{ (t,x) \in \mathbb{R}_{\geq 0} \times \mathbb{R}^n : B(t,x) \leq 0 \right\}.
\end{align}  
\end{enumerate}
\end{problem}

Note that the property in \ref{item:starstar} requires the computation of the solutions to \eqref{eq.1}. However, depending on the regularity of the function $B$ and of the map $F$, as shown in \cite{Sanfelice:monotonicity}, it is possible to use the following infinitesimal conditions that are necessary and sufficient to conclude \ref{item:starstar}.  

\begin{itemize}
\item When $B$ is continuously differentiable, \ref{item:starstar} is satisfied if
\begin{equation}
\label{eq.2c1+}
\begin{aligned} 
\langle \nabla B (t,x) &, [1 \quad  \eta^\top]^\top \rangle \leq 0  
\\ &  \forall \eta \in F(x), ~  \forall (t,x) \in  U(K) \backslash K.
\end{aligned}
\end{equation}
When additionally $F$ is locally Lipschitz, \eqref{eq.2c1+} is equivalent to \ref{item:starstar}.

\item When $B$ is locally  Lipschitz and $F$ is locally bounded, \ref{item:starstar} is satisfied if 
\begin{equation}
\label{eq.2c1lip}
\hspace{-0.4cm}
\begin{aligned} 
\langle \zeta, [1 \quad \eta^\top]^\top \rangle \leq 0 
 \quad & \forall \zeta \in \partial_C B (t,x), ~ \forall \eta \in F(x), \\ &  \forall (t,x) \in U(K) \backslash K,
\end{aligned}
\end{equation}
where $\partial_C B$ is the Clarke generalized gradient of $B$ (see Definition \ref{defgen}). When additionally $F$ is locally Lipschitz and $B$ is regular according to Definition \ref{def.reg}, \eqref{eq.2c1lip} is equivalent to \ref{item:starstar}, 

\item When $B$ is only continuous and $F$ is locally Lipschitz with closed images, \ref{item:starstar} is satisfied if and only if
\begin{equation}
\label{eq.2c1lsc}
\hspace{-0.4cm}
\begin{aligned} 
\langle \zeta, [1 \quad \eta^\top]^\top \rangle \leq 0 \quad 
 & \forall \zeta \in \partial_P B (t,x),~ \forall \eta \in F(x), \\ &  \forall (t,x) \in U(K) \backslash K,
\end{aligned}
\end{equation}
where $\partial_P B$ is the proximal subdifferential of $B$ (see Definition \ref{def-prox}).
\end{itemize}

To solve Problem \ref{prob1}, we start showing that having a time-varying barrier function candidate verifying \ref{item:starstar} is enough to conclude that the system in 
\eqref{eq.1} is safe with respect to $(X_o,X_u)$. In particular, note that \ref{item:starstar} reduces to \ref{item:star} when $B$ is time-independent.

\begin{theorem}  \label{thm2pre} 
Given initial and unsafe sets $(X_o, X_u) \subset \mathbb{R}^n \times \mathbb{R}^n$,  system \eqref{eq.1} is safe with respect to 
$(X_o, X_u)$ if there exists a  lower semicontinuous time-varying barrier function candidate $B : \mathbb{R}_{\geq 0} \times \mathbb{R}^n \rightarrow \mathbb{R}$ such that \ref{item:starstar} holds.
\end{theorem} 

\begin{proof}
Consider the extended system 
\begin{align} \label{extsys}
(\dot{t}, \dot{x}) \in (1, F(x)) \qquad  
(t,x) \in \mathbb{R}_{\geq 0} \times \mathbb{R}^n
\end{align}
and the extended initial and unsafe sets 
$X_{oa} := \mathbb{R}_{\geq 0} \times X_o$ and  
$X_{ua} := \mathbb{R}_{\geq 0} \times X_u$, respectively. 
To use a contradiction argument, we assume that there exists 
a solution $\phi_a := (t,\phi)$ starting from $\phi_{ao} := (0,x_o) \in X_{oa}$ that reaches the set $X_{ua}$ in finite time. 
This implies,  \blue{using the continuity of $\phi_a$},  the existence of $0 \leq t_1 < t_2$ such that $\phi_a ([t_1,t_2],\phi_{ao}) \subset U(\partial K) \backslash \mbox{int}(K)$,  $\phi_a(t_1, x_{ao}) \in \partial K$, and $\phi_a (t_2,\phi_{ao}) \in  U(\partial K) \backslash K$. Hence, $B(\phi_a(t_1,\phi_{ao})) \leq 0$ and $B(\phi_a(t_2,\phi_{ao})) > 0$. However, this contradicts \ref{item:starstar}.
\end{proof}

The challenge in Problem \ref{prob1} is to prove the reverse direction of the statement in Theorem \ref{thm2pre}, namely, necessity of the existence of a barrier function when \eqref{eq.1} is safe. In Section \ref{sec.3}, we prove that result inspired by the converse Lyapunov theorems for 
(non-asymptotic) stability in \cite{Persidskiipap, kurzweil1955, kurzweil1957}.  

\section{Solutions to the Converse Safety Problem} 
\label{sec.3} 

Given the differential inclusion in \eqref{eq.1}, we consider the following mild condition on 
$F$.
\begin{assumption}  \label{item:difinc}
The map $F: \mathbb{R}^n \rightrightarrows \mathbb{R}^n $ is upper semicontinuous, and $F(x)$ is compact and convex for all 
$x \in \mathbb{R}^n$.
\end{assumption}
Assumption \ref{item:difinc} is used in the literature to assure existence of solutions and adequate structural properties for the set of solutions to differential inclusions; see \cite{aubin2012differential, Aubin:1991:VT:120830, clarke2008nonsmooth}. When $F$ is single valued, Assumption \ref{item:difinc} reduces to continuity of $F$. 

\begin{remark} \label{remsetval}
In some of the existing literature, e.g. \cite{goebel2012hybrid}, Assumption \ref{item:difinc} is replaced by the equivalent assumption stating that $F$ needs to be outer semicontinuous  and locally bounded with convex images. Outer semicontinuous and locally bounded set-valued maps are upper semicontinuous with compact images \cite[Theorem 5.19]{rockafellar2009variational}. The converse is also true using \cite[Lemma 5.15]{goebel2012hybrid} and the fact that upper semicontinuous set-valued maps with compact 
images are locally bounded. 
\end{remark}
 
Next, we define the concept of backward solutions to \eqref{eq.1}. 

\begin{definition} [Backward solutions to \eqref{eq.1}] A function $\psi : \dom \psi \rightarrow \mathbb{R}^n$ starting from $x_o \in \mathbb{R}^n$ is a backward solution to \eqref{eq.1} if there exists a solution $\phi$ in the sense of Definition \ref{def.CS}, starting from $x_o$, to the system
\begin{align} \label{eq.1back} 
\dot{x} \in - F(x) \qquad x \in \mathbb{R}^n
\end{align}
such that $\dom \phi = - \dom \psi$ and $\psi(t,x_o) = \phi(-t,x_o)$ for all $t \in \dom \psi$. 
\end{definition}

Furthermore, for the system in \eqref{eq.1}, we introduce the reachability map $R : \mathbb{R} \times \mathbb{R}^n \rightrightarrows \mathbb{R}^n$ as follows:
\begin{itemize}
\item For each 
$(t,x) \in \mathbb{R}_{\geq 0} \times \mathbb{R}^n$,  
\begin{align}
\hspace{-0.4cm} R(t,x) := \{ \phi(s,x): \phi \in \mathcal{S}(x), ~ 
s \in \dom \phi \cap [0,t] \}, \label{eq.R} 
\end{align}
\item For each $(t,x) \in \mathbb{R}_{< 0} \times \mathbb{R}^n$,
\begin{align}
\hspace{-1cm} R(t,x) :=  \{ \phi(s,x): \phi \in \mathcal{S}^{\textrm{back}}(x), ~ 
s \in \dom \phi \cap [t,0] \}, \label{eq.Rbw} 
\end{align}
\end{itemize}
where $\mathcal{S}(x)$ is the set 
of maximal solutions to \eqref{eq.1} starting from $x$ and $\mathcal{S}^{\textrm{back}}(x)$ is the set of maximal backward solutions to \eqref{eq.1} starting from $x$. In simple words, when $t \geq 0$, the set $R(t,x)$ includes all the elements reached by the solutions to \eqref{eq.1} starting from $x$ over the interval $[0,t]$. Similarly, when $t < 0$, the set $R(t,x)$ includes all the elements reached by the backward solutions to \eqref{eq.1} starting from $x$ over interval $[t,0]$. 

Finally, given system \eqref{eq.1} and a set $X_o \subset \mathbb{R}^n$, we introduce the scalar function $B$ defined for each $(t,x) \in \mathbb{R}_{\geq 0} \times \mathbb{R}^n$ by 
\begin{align} 
 B(t,x) = \inf \{|y|_{X_o} : y \in R(-t,x) \}. \label{eqbarup} 
\end{align}
Note that the function  $B$ in \eqref{eqbarup} is inspired by the converse Lyapunov stability theorem in \cite{Persidskiipap}.   As we show in this section, when system \eqref{eq.1} is safe with respect to $(X_o,X_u)$, the function $B$ in \eqref{eqbarup} becomes a time-varying barrier function candidate with respect to 
$(X_o,X_u)$ in the sense of Definition \ref{defTVBf}. Furthermore, we also show that the scalar function $B$ in \eqref{eqbarup} inherits the regularity properties of the reachability map $R$.

\subsection{When $F$ Satisfies Assumption \ref{item:difinc}} \label{sec.3a}

In the following result, for system \eqref{eq.1} satisfying Assumption \ref{item:difinc}, we show that the reachability map $R$ is outer semicontinuous, locally bounded, and continuous with respect to time. A proof is in the appendix.

\begin{proposition} \label{propout}
Suppose that the system in \eqref{eq.1} is forward complete and $F$ satisfies Assumption \ref{item:difinc}. Then,
\blue{\begin{enumerate}
\item The map $R$ is outer semicontinuous and 
locally bounded on $\mathbb{R}_{\geq 0} \times \mathbb{R}^n$.
\item The map $t \mapsto R(t,x)$ is continuous on $\mathbb{R}_{\geq 0}$,  for all $x \in \mathbb{R}^n$.
\end{enumerate}}
\end{proposition} 

Along the lines of \cite[Theorem 1.4.16]{aubin2009set}, given a set-valued map $\Pi : \mathbb{R}^m \rightrightarrows \mathbb{R}^n$ and a set $X \subset \mathbb{R}^n$, we show how the marginal function $f : \mathbb{R}^m \rightarrow \mathbb{R}$ given by
\begin{align} \label{eq.margin}
f(z) := \inf \{ |y|_X : y \in \Pi(z) \} 
\end{align}
inherits the regularity of the set-valued map $R$. A proof is in the Appendix. 

\begin{lemma} \label{lem3}
Consider a locally bounded set-valued map $\Pi : \mathbb{R}^m \rightrightarrows \mathbb{R}^n$ such that $\Pi(z)$ is nonempty for all $z \in \mathbb{R}^m$. Consider a closed and nonempty set $X \subset \mathbb{R}^n$ and the marginal function 
$f : \mathbb{R}^m \rightarrow \mathbb{R}$ in \eqref{eq.margin}.
The following hold:
 \begin{enumerate}
\item If $\Pi$ is outer semicontinuous, then $f$ is lower semicontinuous. 
\item If $\Pi$ is lower semicontinuous, then $f$ is upper semicontinuous.
\item If $\Pi$ is locally Lipschitz, then so is $f$.
\end{enumerate}
\end{lemma} 

The following result is a direct consequence of Proposition \ref{propout} and Lemma \ref{lem3}.

\begin{proposition} \label{cor1up}
Suppose the system in \eqref{eq.1} is \blue{backward complete} and that $F$ satisfies Assumption \ref{item:difinc}. Consider a closed set $X_o \subset \mathbb{R}^n$ and the function $B$ in \eqref{eqbarup}. The following hold:
\begin{enumerate}
\item The function $B$ is lower semicontinuous.
\item The map $t \mapsto B(t,x)$ is continuous.
\end{enumerate} 
\end{proposition}

\begin{proof}
The backward solutions to 
\eqref{eq.1} starting from $x \in \mathbb{R}^n$ are the forward solutions to \eqref{eq.1back} starting from $x$.   Furthermore, having  $F$ satisfying Assumption \ref{item:difinc} implies that $-F$ satisfies Assumption \ref{item:difinc}. Hence, using Proposition \ref{propout}, we conclude that the reachability map $R$ is outer semicontinuous and locally bounded. Next, using the first item in Lemma \ref{lem3}, 
we conclude that $B$ is lower semicontinuous. 

Furthermore, using Proposition \ref{propout}, we conclude that the map $t \mapsto R(t,x)$ is continuous; hence, lower semicontinuous. Next, using the second item in Lemma \ref{lem3}, we conclude that  the map $t \mapsto B(t,x)$ is upper semicontinuous. Finally, since we already showed that $B$ is lower  semicontinuous, we conclude that $t \mapsto B(t,x)$ is continuous.
\end{proof}

We are now ready to provide a solution to Problem \ref{prob1} when $F$ satisfies Assumption \ref{item:difinc}.

\begin{theorem}  \label{thm2}
Suppose the system in \eqref{eq.1} is \blue{backward complete} and $F$ satisfies Assumption
\ref{item:difinc}. Consider initial and unsafe sets 
$(X_o, X_u) \subset \mathbb{R}^n \times \mathbb{R}^n$ such that $X_o$ is closed and $X_o \cap X_u = \emptyset$.  
System \eqref{eq.1} is safe with respect to $(X_o, X_u)$ if and only if there exists a lower semicontinuous time-varying barrier function candidate $B : \mathbb{R}_{\geq 0} \times \mathbb{R}^n \rightarrow \mathbb{R}$, with $t \mapsto B(t,x)$ continuous, such that \ref{item:starstar} holds.
\end{theorem} 

\begin{proof} 
The sufficiency part follows using  Theorem \ref{thm2pre}. To prove the necessary part, we use the barrier function candidate $B$ in \eqref{eqbarup}. Since the set $X_o$ is closed and system \eqref{eq.1} is safe and the system \eqref{eq.1} is forward complete, we conclude that the backward solutions to \eqref{eq.1} starting from $x \in X_u$ will neither reach nor converge to the set $X_o$ in finite time; hence, $B(t,x) > 0$ for all $(t,x) \in \mathbb{R}_{\geq 0} \times X_u$. Then, \eqref{eq.2a+} holds. Furthermore, \eqref{eq.2b+} is trivially satisfied under \eqref{eqbarup}. Then, $B$ is a time-varying barrier function candidate for safety with respect to $(X_o,X_u)$.  Next, we show that the barrier function candidate $B$ is monotonically nonincreasing along the solutions to \eqref{eq.1}. Indeed, consider a solution $\phi : [t, t+h] \rightarrow \mathbb{R}^n$ to \eqref{eq.1} starting from $x_o \in \mathbb{R}^n$ at $t = 0$, for some $h > 0$. Note that $B(t+h , \phi(t+h,x_o)) = \inf \{|y|_{X_o} : y \in R(-t-h,\phi(t+h,x_o))\}$. Furthermore, we use the fact that $R(-t,\phi(t,x_o)) \subset R(-t-h,\phi(t+h,x_o))$, which implies that
\begin{align*}
& B(t+h, \phi(t+h,x_o))\\ & 
= \inf \{ |y|_{X_o} : y \in R(-t-h,\phi(t+h,x_o)) \} 
\\ & \leq  \inf \{ |y|_{X_o} : y \in R(-t,\phi(t,x_o)) \} =   B(t,\phi(t,x_o)).
\end{align*}
Hence, the barrier function candidate $B$ does not increase along the solution $\phi$. Hence, \ref{item:starstar} holds. 
Finally, the fact that $B$ is lower semicontinuous and $t \mapsto B(t,x)$ is continuous follows from Proposition \ref{cor1up}. 
\end{proof}

\begin{example}[Example \ref{expcount} revisited]
Consider system \eqref{cont.exp} in Example \ref{expcount} with the sets $(X_o,X_u)$ as in \eqref{eqSetsCounterExp}. According to the proof Theorem \ref{thm2}, the function 
$B : \mathbb{R}_{\geq 0} \times \mathbb{R}^2 \rightarrow \mathbb{R}$ given by  
\begin{align} \label{eq.sert.counter}
B(t,x) = \phi(-t,|x|),
\end{align}
where $\phi$ is the backward solution to $\dot{r} = (r^2/2) \sin^2 (1/r)$ starting from $|x|$, is a lower semicontinuous time-varying barrier function candidate satisfying \ref{item:starstar}. Indeed, the function $B$ in \eqref{eq.sert.counter} coincides with the time-varying barrier function candidate $B$ in \eqref{eqbarup}. Furthermore, the explicit formula of $B$ in \eqref{eq.sert.counter} is given by
 \begin{equation}
\label{eq.sert.counter1}
\begin{aligned} 
& B(t,x) = \\ & 
\hspace{-0.4cm} 
\left\{
\begin{matrix}
0 & \text{if} ~ x = 0,
\\
\frac{1}{k \pi} &  \text{if} ~ \frac{1}{|x|} = k \pi,~  k \in \mathbb{N}^* 
\\
\frac{1}{arc\cot \left( \cot\left( \frac{1}{|x|} \right) - \frac{t}{2} \right) + k \pi } & \text{if} ~ \frac{1}{|x|} \in (k \pi, (k+1) \pi),~  k \in \mathbb{N},
\end{matrix}
\right.
\end{aligned}
\end{equation}
where $\cot$ is the cotangent function and $arc\cot$ is its inverse function; namely, 
$ arc\cot(\cot(x)) = x$ for all $x \in (0,\pi)$.
\end{example}

\subsection{When $F$ is Locally Lipschitz} \label{sec.3b}

For system \eqref{eq.1} with $F$ locally Lipschitz and having closed images, one can use the well-known Filippov Theorem (see Lemma \ref{lemFil} in the appendix) to conclude that the reachability map  $R$ is also locally Lipschitz. In this setting, we have the following result.

\begin{proposition} \label{proplip.1} 
Suppose that system \eqref{eq.1} is forward complete and $F$ satisfies Assumption \ref{item:difinc} and is locally Lipschitz. Then, $R$ is locally Lipschitz on \blue{$\mathbb{R}_{\geq 0} \times \mathbb{R}^n$}.
\end{proposition} 

Next, using Lemma \ref{lem3}, we show that, when $F$ is locally Lipschitz, $B$ in \eqref{eqbarup} is locally Lipschitz. 

\begin{proposition} \label{propup1}
Suppose that the system in \eqref{eq.1} is \blue{backward complete} and $F$ is locally Lipschitz with closed images. Let $X_o \subset \mathbb{R}^n$ be closed. Then, $B$ in \eqref{eqbarup} is locally Lipschitz. 
\end{proposition} 

\begin{proof}
The backward solutions to \eqref{eq.1} starting from $x \in \mathbb{R}^n$ are the solutions to \eqref{eq.1back} starting from $x$. Furthermore, having  $F$ locally Lipschitz and satisfying Assumption \ref{item:difinc} imply that 
$-F$ is locally Lipschitz and satisfies Assumption \ref{item:difinc}.  Hence, using Proposition \ref{proplip.1}, we conclude that the reachability map $R$ is locally Lipschitz. Finally, using the third item in Lemma \ref{lem3}, we conclude that $B$ is locally Lipschitz.
\end{proof}  

We are now ready to present an equivalent characterization of safety solving Problem \ref{prob1} when $F$ is locally Lipschitz.

\begin{theorem}  \label{thm3}
Suppose the system in \eqref{eq.1} is \blue{backward complete},  and $F$ satisfies Assumption \ref{item:difinc} and is locally Lipschitz. Consider the initial and unsafe sets $(X_o, X_u) \subset \mathbb{R}^n \times \mathbb{R}^n$ such that $X_o$ is closed and $X_o \cap X_u = \emptyset$. System \eqref{eq.1} is safe with respect to $(X_o, X_u)$ 
if and only if there exists a locally Lipschitz time-varying barrier function candidate $B : \mathbb{R}_{\geq 0} \times \mathbb{R}^n \rightarrow \mathbb{R}$ such that \ref{item:starstar} holds.
\end{theorem} 

\begin{proof} 
The proof of the sufficient part follows via Theorem \ref{thm2pre}. To prove the necessary part, we consider the barrier function candidate $B$ in \eqref{eqbarup}. The properties in \eqref{eq.2a+}, \eqref{eq.2b+}, and \ref{item:starstar} follow as in the proof of Theorem \ref{thm2}. Finally, using Proposition \ref{propup1} and Lemma \ref{lem3}, we conclude that the candidate $B$ in \eqref{eqbarup} is locally Lipschitz. 
\end{proof} 

In the following result, we provide a characterization of safety that, rather than 
using \ref{item:starstar}, uses an equivalent infinitesimal condition.
Before that, we first introduce the following lemma relating monotonicity of 
$B$ to infinitesimal inequalities. 

\begin{lemma} \label{lemmonotonicity}
Suppose the system in \eqref{eq.1} is such that $F$ satisfies Assumption \ref{item:difinc} and is locally Lipschitz.  
Let $B : \mathbb{R}^n \rightarrow \mathbb{R}$ be lower semicontinuous. 
Then, given an open set $O \subset \mathbb{R}^n$, the monotonicity property 
\begin{enumerate} [label={($\star'$)},leftmargin=*]
\item \label{item:star'} 
Along each solution $\phi$ starting from $x_o \in \mathbb{R}^n$ and satisfying $\phi(\dom \phi,x_o) \subset O$, the map $t \mapsto B(\phi(t,x_o))$ is nonincreasing; 
\hfill $\bullet$
\end{enumerate}
is satisfied if and only if
\begin{align} \label{eqqqbisss1}
\hspace{-0.2cm} \langle \zeta , \eta  \rangle \leq 0  \qquad \forall \zeta \in \partial_P B(x),~\forall \eta \in F(x),~\forall x \in O. 
\end{align}
\end{lemma}

Lemma \ref{lemmonotonicity} is a particular case of \cite[Corollary 4.13]{Sanfelice:monotonicity}.

\begin{corollary} \label{corf}
Suppose the system in \eqref{eq.1} is \blue{backward complete} and $F$ satisfies 
Assumption \ref{item:difinc} and is locally Lipschitz. Consider initial and unsafe sets $(X_o, X_u) \subset \mathbb{R}^n \times \mathbb{R}^n$ such that $X_o$ is closed and $X_o \cap X_u = \emptyset$. System \eqref{eq.1} is safe with respect to $(X_o, X_u)$ if and only if there exists a locally Lipschitz time-varying barrier function candidate $B : \mathbb{R}_{\geq 0} \times \mathbb{R}^n \rightarrow \mathbb{R}$ such that 
\begin{equation}
\label{eqconvLip}
\begin{aligned} 
 \langle \zeta, [1 \quad \eta^\top]^\top \rangle \leq 0 \quad  & \forall \zeta \in \partial_P B (t,x),~ \forall \eta \in F(x), \\ & \forall (t,x) \in \mathbb{R}_{\geq 0} \times \mathbb{R}^n.
\end{aligned}
\end{equation}
\end{corollary} 

\begin{proof}
According to Theorem \ref{thm3}, safety with respect to $(X_o, X_u)$, when the set $X_o$ is closed, is equivalent to the existence of 
a locally Lipschitz time-varying barrier function $B : \mathbb{R}_{\geq 0} \times \mathbb{R}^n \rightarrow \mathbb{R}$ 
satisfying \eqref{eq.2a+}, \eqref{eq.2b+}, 
and \ref{item:starstar}. Moreover, according to the proof of Theorem \ref{thm2}, for each solution $\phi$ to \eqref{eq.1} starting from $x_o \in \mathbb{R}^n$, 
the map $t \mapsto B(t,\phi(t,x_o))$ is nonincreasing. This property is equivalent to saying that property \ref{item:star'} in 
Lemma \ref{lemmonotonicity} is satisfied while replacing $(x,O)$ therein by 
$((t,x), (\mathbb{R} \times \mathbb{R}^n))$, which completes the proof since the function 
$B$ is continuous. 
\end{proof} 

\begin{example} [Example \ref{expcount} revisited]
Consider system \eqref{cont.exp} in Example \ref{expcount} with the sets $(X_o,X_u)$ as in \eqref{eqSetsCounterExp}. Since the right-hand side in \eqref{cont.exp} is locally Lipschitz and $X_o$ is closed, we conclude via Theorem \ref{thm3} that the time-varying barrier function $B$ in \eqref{eq.sert.counter} is locally Lipschitz and satisfies \eqref{eqconvLip}.
\end{example}

\subsection{When $F$ is Single Valued and Smooth} \label{sec.3c}

In the context of (non-asymptotic) stability of the origin, Kurzweil in \cite{kurzweil1955} deduced from the Lyapunov function constructed in \cite{Persidskiipap}, which is similar to \eqref{eqbarup}, the existence of a Lyapunov function that is $\mathcal{C}^1$ everywhere (except at the origin) under continuous differentiability of $F$ and using the fact that the origin is an equilibrium point. The compactness of the origin is an important requirement for the proof in \cite{kurzweil1955} to hold. Unfortunately, this assumption does not hold when a generic (not necessarily invariant) set $X_o$ is considered instead of the origin, as $X_o$ might be unbounded. To handle this situation, we extend 
\cite[Lemma 48.3]{hahn1967stability} via 
Lemma \ref{lem2} and Lemma \ref{lem1smooth} below, whose proofs are in the Appendix. 
 
\begin{lemma} \label{lem2}
Consider a continuous function $h : \mathbb{R}_{\geq 0} \times \mathbb{R}^n \rightarrow \mathbb{R}_{\geq 0}$ and a closed set $K \subset \mathbb{R}^n$. Assume that
\begin{enumerate}[label={\roman*)},leftmargin=*]
\item \label{item:S11} The function $h$ is positive definite with respect to $K$ uniformly in $t$; namely, $h(t,x) = 0$ for all $t \geq 0$ if and only if $x \in K$,
\item \label{item:S12} The map $t \mapsto h(t, x)$ is nonincreasing for each $x \in \mathbb{R}^n$. 
\end{enumerate}
Then, for any compact set $\mathcal{I} \subset \mathbb{R}^n$ such that $\mathcal{I} \cap K = \emptyset$, there exists a continuous function $g : \mathbb{R}_{\geq 0} \times \mathcal{I} \rightarrow \mathbb{R}_{\geq 0}$ such that
\begin{enumerate}[label={\arabic*)},leftmargin=*]
\item \label{item:S21} The function $ g \in \mathcal{C}^1( \mathbb{R}_{\geq 0} \times \mbox{int}(\mathcal{I})) $,
\item \label{item:S22} For any $(t,x) \in \mathbb{R}_{\geq 0} \times \mathcal{I}$,
\begin{align} \label{eq.eqnc3}
\frac{1}{2} h(t,x) \leq g(t,x) \leq 2 h(t,x),
\end{align}
\item \label{item:S23} The map $t \mapsto g(t, x)$ is nonincreasing for each $x \in \mathcal{I}$.
\end{enumerate}
\end{lemma}

\begin{lemma} \label{lem1smooth}
Consider a continuous function $h : \mathbb{R}_{\geq 0} \times \mathbb{R}^n \rightarrow \mathbb{R}_{\geq 0}$ and  consider a closed set $K \subset \mathbb{R}^n$. Assume that \ref{item:S11}-\ref{item:S12} in Lemma \ref{lem2} hold. Then, there exists a continuous function $g : \mathbb{R}_{\geq 0} \times \mathbb{R}^n \rightarrow \mathbb{R}_{\geq 0}$ such that
\begin{enumerate}[label={\arabic*)},leftmargin=*]
\item \label{item:S41} The function $ g \in \mathcal{C}^1( \mathbb{R}_{\geq 0} \times (\mathbb{R}^n\backslash K))$,
\item \label{item:S42} For all $(t,x) \in \mathbb{R}_{\geq 0} \times \mathbb{R}^n$,
\begin{align} \label{eq.eqnc}
\frac{1}{2} h(t,x) \leq g(t,x) \leq 2 h(t,x),
\end{align}
\item \label{item:S43} The map $t \mapsto g(t, x)$ is nonincreasing for each 
$x \in \mathbb{R}^n$.
\end{enumerate}  
\end{lemma} 
 
It should be added that the origin being an equilibrium plays an important role in \cite{kurzweil1955} to guarantee positive definiteness of a certain function constructed in the proof. However, such a function is not necessarily positive definite when the origin is replaced by a generic closed set. To handle this situation, we propose a state dependent change in the time scale such that, in the new time scale, this function becomes positive definite.

\begin{theorem} \label{prop12}
Suppose the system in \eqref{eq.1} is \blue{backward complete},  and $F$ is single valued and continuously differentiable. Consider initial and unsafe sets 
$(X_o, X_u) \subset \mathbb{R}^n \times \mathbb{R}^n$ such that $X_o$ is closed and $X_o \cap X_u = \emptyset$.  
System \eqref{eq.1} is safe with respect to $(X_o, X_u)$ if and only if there exists a continuous time-varying barrier function candidate $B:\mathbb{R}_{\geq 0} \times \mathbb{R}^n \rightarrow \mathbb{R}$ of class 
$\mathcal{C}^1\left( (\mathbb{R}_{\geq 0} \times \mathbb{R}^n) \backslash K \right)$, where 
$K := \left\{ (t,x) \in \mathbb{R}_{\geq 0} \times \mathbb{R}^n : B(t,x) \leq 0  \right\}$,  such that

\begin{equation}
\label{eq.2bis1}
\begin{aligned} 
\langle \nabla B(t,x), [1 \quad F(x)^\top]^\top \rangle & \leq 0  \\ & \forall (t,x) \in \left( \mathbb{R}_{\geq 0} \times \mathbb{R}^n \right) \backslash K. 
\end{aligned} 
\end{equation}
\end{theorem}

\begin{proof}
In order to prove the sufficient part of the statement, we use a contradiction. That is, assume the existence of a solution $\phi$ starting from $x_o \in X_o$ such that $\phi(T, x_o) \in X_u$ for some $T>0$. The latter fact implies, using \eqref{eq.2bis1}, that $B(0,\phi(0,x_o)) \leq 0$ and $B(T, \phi(T,x_o)) > 0$. Furthermore, since $t \mapsto B(t,\phi(t,x_o))$ is continuous, we also conclude the existence of $0 \leq T_1 < T$ such that 
$B(T_1, \phi(T_1,x_o)) = 0$ and $B(t,\phi(t,x_o)) > 0$ for all $t \in (T_1,T]$. Hence, $ B(T, \phi(T,x_o)) - B(T_1, \phi(T_1,x_o)) > 0$ and using the continuity of $t \mapsto B(t,\phi(t,x_o))$, we also conclude the existence of $\epsilon > 0$ sufficiently small such that $T_1 + 2\epsilon < T$ and $ B(T-\epsilon, \phi(T-\epsilon,x_o)) - B(T_1+\epsilon, \phi(T_1+\epsilon,x_o)) > 0$.  However, since $B(t,\phi(t,x_o)) > 0$ for all $t \in (T_1,T]$, it follows that $t \mapsto B(t,\phi(t,x_o)) $ is $\mathcal{C}^1((T_1, T))$. Hence, 
\begin{align*}
& B(T-\epsilon, \phi(T-\epsilon,x_o)) - B(T_1+\epsilon, \phi(T_1+\epsilon,x_o)) = \\  & \int^{T-\epsilon}_{T_1+\epsilon} \frac{\partial B}{\partial t}(t,\phi(t,x_o)) + \frac{\partial B}{\partial x}(t,\phi(t,x_o)) F(\phi(t,x_o)) dt \leq 0,
\end{align*}
which yields to a contradiction. 
\\

To prove the necessary part, we first propose to render the behavior of the system \eqref{eq.1} around the set $X_o$ similar to the behavior of a smooth system around its equilibrium point. More precisely, by proposing a new time scale, we render the set $X_o$ unreachable in finite time by the solutions starting from $\mathbb{R}^n \backslash X_o$. To this end, given an initial condition $x_o \in \mathbb{R}^n \backslash X_o$, we propose the following new time scale
\begin{align} \label{new-timesc}
\tau (t,x_o) := t + \int^{t}_{0} \frac{1}{V(\phi(s,x_o))} ds,
\end{align}  
where $V$ is any locally Lipschitz and positive definite function with respect to the set $X_o$ which is differentiable everywhere outside the set $X_o$ and $\phi$ is the solution to \eqref{eq.1} starting from $x_o$. The function $V$ always exists for any given closed set $X_o \subset \mathbb{R}^n$ and it can be constructed using Lemma \ref{lem1smooth} by considering the distance function with respect to $X_o$ to be the function $h$ therein. 
Furthermore, we let $ \psi(\tau(t,x_o), x_o) := \phi(t,x_o)$. 

As a consequence, the derivative of $y$ with respect to the new time scale $\tau$ satisfies 
\begin{equation}
\label{new-timesys}
\begin{aligned} 
\psi'(\tau,x_o) & :=  \frac{d \psi}{d \tau}(\tau,x_o)  = \frac{d \phi}{d \tau}(t,x_o) \\ & 
= \frac{F(\phi(t,x_o))}{\frac{d\tau}{dt}(t,x_o)} =  \frac{F(\phi(t,x_o)) V(\phi(t,x_o))}
{1 + V(\phi(t,x_o))}.
\end{aligned}
\end{equation}
Hence, 
\begin{align} \label{new-timesys1}
\psi'(\tau,x_o) = \frac{F(\psi(\tau, x_o)) V(\psi(\tau, x_o))}{1 + V(\psi(\tau, x_o))}.
\end{align}
Note that the solutions to the system 
\begin{align} \label{new-timesys1bis}
\psi' = \frac{F(\psi) V(\psi)}{1 + V(\psi)}
\end{align}
starting from $x_o$ cannot reach $X_o$ when starting outside that set $X_o$. Moreover, the set $X_o$ is forward invariant under the system \eqref{new-timesys1bis}.
 
Let us now introduce the continuous function $h: \mathbb{R}_{\geq 0} \times \mathbb{R}^n \rightarrow \mathbb{R}_{\geq 0}$ as 
\begin{align} \label{new-timesys2}
h (\tau , x_o) := \inf \{|y|_{X_o} : y \in R(\tau,x_o) \},
\end{align}
where $R$ in this case is the reachable set along the solutions to the system \eqref{new-timesys1bis}. 
Using Proposition \ref{propup1}, we conclude that the function $h$ is locally Lipschitz. Furthermore,
since the right-hand side of \eqref{new-timesys1bis} is locally Lipschitz, we conclude that $h$ positive definite with respect to the set $X_o$. Finally, the map $\tau \mapsto h(\tau, x_o)$ 
non-increasing for all $x_o \in \mathbb{R}^n$. Therefore, using Lemma \ref{lem1smooth}, we conclude the existence of a continuous function 
$g : \mathbb{R}_{\geq 0} \times \mathbb{R}^n \rightarrow \mathbb{R}_{\geq 0}$ which is $\mathcal{C}^1$ outside the set $X_o$, non-increasing with respect to the first argument, and satisfies
\begin{align*}
\frac{1}{2} h (\tau , x_o) \leq g (\tau , x_o) \leq 2 h (\tau , x_o) \quad \forall 
(\tau, x_o) \in (\mathbb{R}_{\geq 0} \times \mathbb{R}^n). 
\end{align*}
Next, we introduce the barrier candidate 
$B: \mathbb{R}_{\geq 0} \times \mathbb{R}^n \rightarrow \mathbb{R}_{\geq 0}$ as
\begin{align} \label{eq.cert}
B(t,x) := \left\{ 
\begin{matrix}
g(\tau (t,\chi(-t,x)), \chi(-t,x)) & 
\\ & \hspace{-2cm} \mbox{if}~ \chi([-t,0],x) \cap X_o = \emptyset, 
\\
 0 & \mbox{otherwise},
\end{matrix}
\right.
\end{align}
where $\chi$ is the backward solution 
to \eqref{eq.1} starting from $x$. 
Note that when $x \in X_u$, for each $t \geq 0$, we have $\chi([-t,0],x) \cap X_o = \emptyset$, hence, 
\begin{align*}
B(t,x)  = & ~ g(\tau (t,\chi(-t,x)), \chi(-t,x)) \\ \geq & ~  h(\tau (t, \chi(-t,x)), \chi(-t,x))/2 > 0.
\end{align*}
Contrary, when $x \in X_o$, $\chi([-t,0],x) \cap X_o \neq \emptyset$, hence, $B(t,x)=0$. Furthermore, we show that the candidate $B$ is non-increasing along the solutions to \eqref{eq.1} by showing that 
$$ B(t+h, \phi(t+h,x_o)) \leq B(t,\phi(t,x_o)) \quad \forall t \geq 0, \quad \forall h \geq 0, $$ 
and for each solution $\phi$ to \eqref{eq.1} starting from $x_o$. To this end, we distinguish two complementary situations. 
\begin{enumerate}
\item When $\chi([-(t+h),0],\phi(t+h,x_o)) \cap X_o = \emptyset$, it follows that  $B(t+h, \phi(t+h,x_o)) = 0 \leq B(t,\phi(t,x_o))$.  
\item When $\chi([-(t+h),0], \phi(t+h,x_o)) \cap X_o = \emptyset$, it follows that 
\begin{align*} 
B(t+h, &~ \phi(t+h,x_o)) = g(\tau(t+h,x_o), x_o) \\ 
& = g(\tau (t+h, \chi(-t,\phi(t,x_o))), 
\chi(-t,\phi(t,x_o))) \\ & \leq 
g(\tau(t, \chi(-t,\phi(t,x_o))), \chi(-t,\phi(t, x_o))) \\ & = B(t,\phi(t,x_o)).
\end{align*}  
To obtain the latter inequality, we used the fact that the function $g$ is non-increasing with respect to its first argument uniformly in the second. 
\end{enumerate}

In order to complete the proof, it remains to show that $B \in \mathcal{C}^1\left( (\mathbb{R}_{\geq 0} \times \mathbb{R}^n) \backslash K \right)$. Indeed, for 
$(t,x) \in (\mathbb{R}_{\geq 0} \times \mathbb{R}^n) \backslash K$, we have $B(t,x) > 0$. Hence, $\chi([-t,0],x) \cap X_o = \emptyset$ and 
$B(t, x) = g(\tau (t,\chi(-t,x)), \chi(-t,x))$.  Furthermore, since the function $B$ is continuous, we conclude the existence of $U(t,x)$ an open neighborhood around $(t,x)$ such that, for any $(t',x') \in U(t,x)$,  we have $ B(t', x') = g(\tau(t',\chi(-t',x')), \chi(-t',x')) > 0$. 
Next, we note that the map $(\tau, x) \mapsto g(\tau, x)$ is continuously differentiable on $\mathbb{R}_{\geq 0} \times (\mathbb{R}^n \backslash X_o)$. Furthermore, $\chi(-t,x)$ is continuously differentiable with respect to its arguments since $F \in \mathcal{C}^1$, see \cite[Chapter 5]{sideris2013ordinary}. Moreover, the map $(t,x) \mapsto \tau(t,\chi(-t,x))$ is $\mathcal{C}^1$ provided that $\chi(-s,x) \notin X_o$ for all $s \in [0,t]$, which completes the proof. 
\end{proof} 

\begin{example}[Example \ref{expcount} revisited]
Consider system \eqref{cont.exp} in Example \ref{expcount} with the sets $(X_o,X_u)$ as in \eqref{eqSetsCounterExp}. Since the right-hand side in \eqref{cont.exp} is continuously differentiable, we conclude using Theorem \ref{prop12} that system \eqref{cont.exp} admits a continuous time-varying barrier function candidate $B:\mathbb{R}_{\geq 0} \times \mathbb{R}^n \rightarrow \mathbb{R}$ of class $\mathcal{C}^1\left( (\mathbb{R}_{\geq 0} \times \mathbb{R}^n) \backslash K \right)$ satisfying  \eqref{eq.2bis1}. In particular, the function $B$ in \eqref{eq.sert.counter}, given explicitly in \eqref{eq.sert.counter1}, corresponds to such a smooth barrier function.
\end{example}

\section{Connections to Results in the Literature} \label{sec.4-}

\subsection{Connections to Tangent-Cone-Type Conditions} \label{sec.apNagumo}

According to \cite{nagumo1942lage}, given a closed set $K \subset \mathbb{R}^n$, when the solutions to \eqref{eq.1} are unique or when $F$ is locally Lipschitz according to Definition \ref{deflip}, the set $K$ is forward pre-invariant if and only if
\begin{align} \label{eq.Nagumo}
F(x) \subset T_K(x) \qquad \forall x \in \partial K.
\end{align}
Note that \eqref{eq.Nagumo} involves the contingent cone $T_K$ and the map $F$ on the boundary of the closed set $K$. However, in general, as stressed in \cite{Aubin:1991:VT:120830}, invariance depends on the values of $F$ outside $K$ rather than on its boundary. Under mild regularity properties on $F$, the external contingent cone $E_K$ is used in \cite{Aubin:1991:VT:120830}, and \eqref{eq.Nagumo} can be replaced by
\begin{align} \label{eq.ExtNagumo}
F(x) \subset E_K(x) \qquad \forall x \in \mathbb{R}^n \backslash K,
\end{align}
where $E_K$ is the \textit{external contingent} cone of $K$ at $x$ is given by $E_K(x) :=  \left\{ v\in \mathbb{R}^n :\liminf_{h \rightarrow 0^+} \frac{|x+hv|_K - |x|_K}{h} \leq 0  \right\}$. 

Note that conditions \eqref{eq.Nagumo} and \eqref{eq.ExtNagumo} resemble conditions \eqref{eq.2c2} and \eqref{eq.2c1}, respectively. Indeed, in \eqref{eq.Nagumo} and \eqref{eq.ExtNagumo}, we are using the distance function $B(x) := |x|_K$ since \eqref{eqaddd} holds. However, since the distance function to a set is only locally Lipschitz, the gradient-based inequalities in \eqref{eq.2c2} and \eqref{eq.2c1} are replaced by the limits in the definitions of $T_K$ and $E_K$ in \eqref{eq.Nagumo} and \eqref{eq.ExtNagumo}, respectively.

\subsection{Connections to Results using Barrier Functions} \label{sub.sec.conn}

Before comparing our results to the existing literature, we recall the following useful notions.

\begin{itemize}
\item \textit{Uniqueness function \cite{draftautomatica}.}  A function $g : \mathbb{R} \rightarrow \mathbb{R}$ is said to be a uniqueness function if, for each continuous function 
$l: \mathbb{R}_{\geq 0} \rightarrow \mathbb{R}_{\geq 0}$ satisfying  
$l(0) = 0$ and, for some $\epsilon > 0$,
\begin{align*} 
\limsup_{h \rightarrow 0^+} \frac{l(t+h)-l(t)}{h} \leq g(l(t)) ~~
\mbox{for a.a.}~ t \in [0, \epsilon],
\end{align*} 
it follows that $l(t) = 0$ for all $t \in [0,\epsilon]$. 

\item \textit{Minimal functions \cite{konda2019characterizing}.} A function $g : \mathbb{R} \rightarrow \mathbb{R}$ is said to be a minimal function if, for each continuous function 
$l: \mathbb{R}_{\geq 0} \rightarrow \mathbb{R}_{\geq 0}$ satisfying  
$l(0) \leq 0$ and, for some $\epsilon > 0$,
\begin{align*} 
\limsup_{h \rightarrow 0^+} \frac{l(t+h)-l(t)}{h} \leq g(l(t)) ~~
\mbox{for a.a.}~ t \in [0, \epsilon],
\end{align*} 
 it follows that $l(t) \leq 0$ for all $t \in [0,\epsilon]$. 

\item \textit{Extended class-$\mathcal{K}$ functions \cite{ames2014control}.}  A continuous function $g : \mathbb{R} \rightarrow \mathbb{R}$ is said to be an extended class-$\mathcal{K}$ function if $g$ is strictly increasing and $g(0) = 0$.
\end{itemize}

The nonpositive sign required in 
\eqref{eq.2c1+} can be relaxed using uniqueness functions or minimal functions, as shown in \cite{draftautomatica} and \cite{konda2019characterizing}, respectively. More precisely, for a general time-varying barrier function candidate $B \in \mathcal{C}^1$, condition \eqref{eq.2c1+} can be relaxed to
\begin{equation}
 \label{eq.2c3}
\begin{aligned} 
\langle \nabla B (t,x) , [1 \quad  \eta^\top]^\top \rangle & \leq g(B(t,x))  
\\ &  \forall \eta \in F(x), ~  \forall (t,x) \in  U(K) \backslash K,
\end{aligned}
\end{equation}
where $g$ is either a uniqueness or a minimal function. Furthermore, given a time-independent barrier function candidate $B \in \mathcal{C}^1$, according to \cite{DAI201762, 10.1007/978-3-642-39799-8_17,ames2014control}, the following condition implies forward pre-invariance of the set $K$ in \eqref{eqaddd}:
\begin{equation}
 \label{eq.2c2p}
\begin{aligned} 
\langle \nabla B(x), \eta \rangle \leq g(B(x))  \quad  \forall \eta \in F(x), ~  \forall x \in  U(K),
\end{aligned}
\end{equation}
where the function $g$ is either an extended class-$\mathcal{K}$ or a locally Lipschitz function. Moreover, when $F$ is locally Lipschitz and the set $K$ is compact, \eqref{eq.2c2p} is equivalent to forward pre-invariance of the set $K$. This converse result, in addition to restricting the class of systems \eqref{eq.1}, requires the existence of a $\mathcal{C}^1$ barrier function candidate. As we show in the following example, it is possible to find situations where the sets $(X_o, X_u)$ do not admit a $\mathcal{C}^1$ time-independent barrier function candidate.  

\begin{example} \label{exp3}
Let $X_u := \mathbb{R}^2 \backslash X_o$ and 
$ X_o := \{x \in \mathbb{R}^2 : \rho_1(x) \leq 0,~\rho_2(x) \leq 0 \}$, where $\rho_1$ and $\rho_2 : \mathbb{R}^n \rightarrow \mathbb{R}$ are $\mathcal{C}^1$ functions such that, for each $i \in \{1,2\}$, $\nabla \rho_i(x) \neq 0$ for all 
$x \in \mathbb{R}^n$ such that $\rho_i (x) = 0$. Furthermore, suppose there exists $x_o \in \mathbb{R}^2$ such that $\rho_1(x_o) = \rho_2(x_o) = 0$ and the vectors 
$\nabla \rho_1(x_o)$ and $\nabla \rho_2(x_o)$ are linearly independent. For this choice of $(X_o,X_u)$, we show that it is not possible to find a $\mathcal{C}^1$ barrier function candidate. To arrive to a contradiction, we assume the existence of $B : \mathbb{R}^n \rightarrow \mathbb{R}$ such that $X_o := \{x \in \mathbb{R}^2 : B(x) \leq 0 \}$.
Assume without loss of generality that $\nabla B(x_o) \neq 0$. Hence, using \cite[Proposition 4.3.7]{aubin2009set}, we conclude that $T_{X_o}(x_o) = \{v \in \mathbb{R}^2 : \langle \nabla B(x_o), v \rangle \leq 0 \}$.
Moreover, from the construction of $X_o$ using $\rho_1$ and $\rho_2$, we conclude that
$$ T_{X_o}(x_o) = \{v \in \mathbb{R}^2 : \langle \nabla \rho_1(x_o), v \rangle \leq 0, ~ \langle \nabla \rho_2(x_o), v \rangle \leq 0 \}. $$
Now, let $y \in \mathbb{R}^2 \backslash \{0\}$ be such that 
$\langle \nabla B(x_o), y \rangle  =  0$; hence, 
$\lambda y \in T_{X_o}(x_o)$ for all $\lambda \in \mathbb{R}$. The latter implies that 
$\langle \nabla \rho_1(x_o), y \rangle = 0$ and 
$\langle \nabla \rho_2(x_o), y \rangle = 0$, 
which contradicts the fact that the vectors $\nabla \rho_1(x_o)$ and $\nabla \rho_2(x_o)$ are linearly independent.
\end{example}

When $g$ is an extended class-$\mathcal{K}$ function, the condition in \eqref{eq.2c2p} is a particular case of \eqref{eq.2c1}, and thus a particular case of \eqref{eq.2c1+}. Furthermore, when $g$ is locally Lipschitz, \eqref{eq.2c2p} becomes a particular case of \eqref{eq.2c3}. Indeed, every locally Lipschitz function is a uniqueness function and a minimal function at the same time. However, \textit{Osgood} functions are examples of uniqueness and minimal functions that are not locally Lipschitz \cite{redheffer1972theorems, agarwal1993uniqueness}. Finally, imposing the inequality in \eqref{eq.2c2p} to hold on $U(K)$ instead of only on $U(K) \backslash K$ is not necessary to guarantee safety; however, it becomes useful when using numerical methods for the (online) design of smooth controllers that enforce safety \cite{jankovic2018robust}. 

\begin{remark}
When the barrier function candidate $B$ is locally Lipschitz, condition \eqref{eq.2c1lip} can also be relaxed using uniqueness functions and minimal functions. Furthermore, when $B$ is time independent, condition \eqref{eq.2c1lip} reduces to the condition used in \cite[Theorem 2]{glotfelter2017nonsmooth} 
and \cite{8625554}.
\end{remark}

\subsection{Connections to Existing Converse Safety Results}
Via the following simple example, we illustrate the limitation of the converse safety results in \cite{prajna2005necessity}, \cite{wisniewski2016converse}, and \cite{ratschan2018converse}. 

\begin{example} \label{exp}
Consider the system  
\begin{align} \label{eqexplast} 
\dot{x} = 
\begin{bmatrix} 
-1 & -10 \\ 1 & 0  
\end{bmatrix} x \qquad x \in \mathbb{R}^2,
\end{align} 
and let the initial and unsafe sets be given by
\begin{align*} 
X_o := \left\{ x \in \mathbb{R}^2 : x_1^2 + x_2^2 \leq 1 \right\}, ~ X_u := \left\{ x \in \mathbb{R}^2 : x_2 \geq 2  \right\}. 
\end{align*}
Note that the set $X_o$ is not forward pre-invariant but the system \eqref{eqexplast} is safe with respect to $(X_o,X_u)$. One way to show this fact consists in verifying \eqref{eq.2c} using the barrier function candidate
$B(x) := x_1^2/10 + x_2^2 - 1$. Note that system \eqref{eqexplast} admits the origin as an equilibrium point, which is a trivial limit cycle. Hence, it is not possible to apply the converse safety result in \cite{prajna2005necessity}. Furthermore, according to the robust safety notion introduced in \cite{wisniewski2016converse}, which is included below, the system \eqref{eqexplast} is robustly safe with respect to $(X_o,X_u)$.  However, the system \eqref{eqexplast} is not defined on a bounded manifold. Hence, it is not possible to use the converse result in \cite{wisniewski2016converse}.  

\begin{definition} [Robust safety \cite{wisniewski2016converse}] \label{defrobsaf}
System \eqref{eq.1} is said to be robustly safe with respect to $(X_o,X_u)$ if there exists $V_o := U(X_o)$ and $V_u := U(X_u) $ such that the vector field $F$ separates $V_o$ from $V_u$.
In turn, a vector field $F$ is said to separate a set $V_o$ from a set $V_u$ if $F$ does not join $V_o$ to $V_u$. In turn, a vector field $F$ is said to join a set $V_o$ to a set $V_u$ if one of the following is true.
\begin{enumerate}
\item  There exists a solution to \eqref{eq.1} starting from $V_o$ that reaches $V_u$. 
\item There is not a succession of singular elements (singular points and limit cycles) $\{\beta_1, \beta_2, ..., \beta_N\}$, $N \in \mathbb{N}$, such that the following properties hold simultaneously:
\begin{itemize}
\item A forward solution to \eqref{eq.1}, starting from $V_o$, converges to $\beta_1$. 
\item A backward solution to \eqref{eq.1}, starting from $V_u$, converges to $\beta_N$.
\item  A broken solution joins $\beta_1$ to $\beta_N$; namely, for each $i \in \{1,2,...,N \}$, there is $x_{oi} \in \mathbb{R}^n$ such that the forward solution to \eqref{eq.1} starting from $x_{oi}$ converges to $\beta_{i+1}$ and the backward solution to \eqref{eq.1} starting from $x_{oi}$ converges to 
$\beta_{i}$; see \cite{wisniewski2016converse} for more details.
\end{itemize}
\end{enumerate}
\end{definition}
Finally, note that the system 
\begin{align} \label{eqexplast1}  
\dot{x} \in
\begin{bmatrix} 
-1 & -10 \\ 1 & 0  
\end{bmatrix} x  + \epsilon \mathbb{B} \qquad 
x \in \mathbb{R}^2
\end{align}
is input-to-state stable (ISS) with respect to $\epsilon$ and the sets $X_o$ and $X_u$ are closed and disjoint. Hence, for $\epsilon>0$ sufficiently 
small, system \eqref{eqexplast1} is safe with respect to $(X_o,X_u)$. However, the complement of the set $X_{u}$ is unbounded. As a result, we cannot use the converse result in \cite{ratschan2018converse}. 
\end{example}

\subsection{Connections to Results on Conditional Invariance} \label{appen1}

According to \cite[Theorem 2]{ladde1974}, the set $X_s$ is conditionally invariant with respect to $X_o$ if there exists a continuously differentiable function $V : \mathbb{R}^n \rightarrow \mathbb{R}$ such that the following three conditions hold:
\begin{enumerate}[label={\roman*)},leftmargin=*]
\item \label{item:S1} For each $x \in \mathbb{R}^n \backslash X_o$ and for each $y_x \in X_o$ satisfying $y_x := \text{arg inf} \{V(x-z) : z \in X_o \}$, we have 
$\langle \nabla V(x-y_x), \eta \rangle \leq 0$ for all $\eta \in F(y_x)$.
\item \label{item:S2} There exists $a \in \mathbb{R}$ such that the function $B: \mathbb{R}^n \rightarrow \mathbb{R}$ given by $B(x) := \inf \{V(x-z) - a : z \in X_o \}$ satisfies  
\begin{align} \label{eqadded} 
B(x) > 0 ~ \forall x \in \partial X_s, ~ 
B(x) \leq 0 ~ \forall x \in \partial X_o. 
\end{align} 
\item \label{item:S3} For each $(x,y) \in (\mathbb{R}^n \backslash X_o) \times X_o$,
\begin{align} \label{eq.3}
\langle \nabla V(x-y) , \eta_x - \eta_y \rangle  \leq 
g(V(x-y) - a)
\end{align}
for all $(\eta_x,\eta_y) \in F(x) \times F(y)$, where the scalar function $g$ is a minimal function (see Section \ref{sub.sec.conn}).
\end{enumerate}
The proof of this result is based on showing that, along the solutions to \eqref{eq.1}, the function $B$ cannot become positive when starting from nonpositive values. Indeed, using \ref{item:S1} and \ref{item:S3}, we can prove that 
\begin{equation}
\label{eqnew} 
\begin{aligned} 
\langle \zeta, \eta \rangle \leq g(B(x)) ~~~\forall \zeta \in  \partial_C B (x),~ \forall \eta \in F(x), ~  \forall x \in \mathbb{R}^n \backslash X_o.
\end{aligned}
\end{equation}
Note that \ref{item:S2} along with \eqref{eqnew} guarantee forward pre-invariance of the set $K$ in \eqref{eqaddd}; however, condition \eqref{eq.2} is not necessarily satisfied in this case. Furthermore, \ref{item:S2} and \eqref{eqnew} imply that $\mbox{int}(X_s)$ is conditionally invariant with respect to $X_o$, when $X_o \subset \mbox{int}(X_s)$. 

Next, we present a result that generalizes \cite[Theorem 2]{ladde1974}. In our result, we distinguish \textit{strict conditional invariance}, where the solutions starting from $X_o$ remain in the interior of $X_s$, from \textit{conditional invariance}, where the solutions starting from $X_o$ remain in $X_s$. To match the setting in \cite{ladde1974}, it is written for a time-independent barrier function candidate.

\begin{theorem} \label{prop1}
Consider the system in \eqref{eq.1} such that $F$ satisfies Assumption \ref{item:difinc}. Let $(X_o,X_s) \subset \mathbb{R}^{n} \times \mathbb{R}^n$ with $X_o \subset X_s$, $g : \mathbb{R} \rightarrow \mathbb{R}$ be a minimal 
function, and $B : \mathbb{R}^n \rightarrow \mathbb{R}$ be locally Lipschitz.
\begin{enumerate}
\item The set $X_s$ is conditionally invariant with respect to $X_o$ if 
\begin{align} \label{eqaddedbis}
\hspace{-0.6cm} B(x) > 0 ~ \forall x \in U(X_s) \backslash X_s, \quad  
B(x) \leq 0 ~ \forall x \in \partial X_o, 
\end{align}
\begin{equation}
\label{eq.12a} 
\begin{aligned} 
\langle \zeta, \eta \rangle \leq  g(B(x)) ~~~ & \forall \zeta \in \partial_C B(x), ~ \forall \eta \in F(x), \\ & \forall x \in U(X_s) \backslash X_o. 
\end{aligned}
\end{equation}

\item The set $X_s$ is strictly conditionally invariant with respect to $X_o \subset \mbox{int}(X_s)$ if \eqref{eqadded} holds and
\begin{equation}
\label{eq.13a} 
\begin{aligned}
\hspace{-0.6cm} 
\langle \zeta, \eta \rangle \leq g(B(x)) ~~ & \forall \zeta \in \partial_C B(x), ~ \forall \eta \in F(x), ~ \forall x \in X_s \backslash X_o.  
\end{aligned}
\end{equation}
\end{enumerate} 
\end{theorem} 

\begin{proof}
To reach a contradiction and establish item 1 (respectively, item 2), we assume that \eqref{eqaddedbis} (respectively, \eqref{eqadded}) holds and $X_s$ is not conditionally invariant (respectively, not strictly conditionally invariant) with respect to $X_o$. That is, there exists a solution $\phi$ starting from $x_o \in \partial X_o$ --- thus, $B(x_o) \leq 0$ --- and there exists $T>0$ such that $\phi(T, x_o) \in U(X_s) \backslash X_s$ (respectively, $\phi(T, x_o) \in \partial X_s$); thus, $\phi(T,x_o) > 0$, and $\phi((0,T], x_o) \subset U(X_s) \backslash X_o$ (respectively, $\phi((0,T], x_o) \subset X_s \backslash X_o$). Hence, according to \cite[Page 7]{sanfelice2007invariance} and \cite{clarke1990optimization}, we conclude that, for almost all $t \in [0,T]$, 
\begin{align*} 
\dot{B}(\phi(t,x_o)) & \leq \sup \{ \langle \zeta, \dot{\phi}(t,x_o) \rangle : \zeta \in \partial_C B(\phi(t,x_o))  \} 
\\ & \leq g(B(\phi(t,x_o))), 
\end{align*} 
with $B(x_o) \leq 0$ and $B(\phi(T,x_o)) > 0$, which yields a contradiction since $g$ is a minimal function, implying that $B(\phi(T,x_o))$ has to be nonpositive.
\end{proof} 

Theorem \ref{prop1} relaxes condition \eqref{eq.2} while assuming that the inequality in \eqref{eq.2c3} holds in a relatively larger set.

\section{Conclusion and Future Work} \label{sec.4}

In this paper, we propose sufficient and necessary conditions for safety in differential inclusions. Guided by the lack of existence of autonomous and continuous barrier functions certifying safety, time-varying barrier functions 
are proposed, and their existence is shown to be both necessary as well as sufficient. The regularity of the proposed time-varying barrier functions depends on the regularity of the right-hand side of the system. 

Future work pertains to solve Problem \ref{prob1} for constrained systems of the form 
\begin{align} \label{eq.1bis}
\dot{x} \in F(x) \qquad x \in C \subset \mathbb{R}^n
\end{align}
or, more generally, hybrid systems as in \cite{goebel2012hybrid}. 
Although the sufficient conditions for safety in constrained and hybrid systems are studied in \cite{draftautomatica},  the converse problem is still not fully answered in the literature.  
Indeed,  the converse safety results in \cite{wisniewski2016converse} and  \cite{prajna2005necessity} consider only particular cases of constrained systems, where the sets $C$, $X_o$, and $X_u$ are assumed to be compact, and $F$ is assumed to be at least continuously differentiable.  Furthermore,  in \cite{wisniewski2016converse},  the system is assumed to admit a \textit{Meyer} function and in \cite{prajna2005necessity} \eqref{eq.2cd} is assumed to hold which, as shown in Example \ref{exp}, are rather restrictive conditions to impose.  It is important to note that Theorem \ref{thm2} can already be extended to constrained and hybrid systems; see \cite{CP5-SIMUL4-ACC2019}.  However,  to establish the existence of a barrier function that is continuous or smooth, the problem becomes more challenging due to the presence of the constraint.   In particular,   the regularity properties of the reachability map $R$, that allows to conclude Lipschitz continuity and continuous differentiability of the marginal functions in \eqref{eqbarup} and \eqref{new-timesys2}, 
are not necessary satisfied in the constrained case; see \cite{ACC2020-1}.

\section{Appendix}

\subsection{Auxiliary Results}

We start this section by introducing the reachability map $R^b : \mathbb{R}_{\geq 0} \times \mathbb{R}^n \rightrightarrows \mathbb{R}^n$, along the solutions to \eqref{eq.1}, given by 
\begin{align}
R^b(t, x) := & \left\{ \phi(s,x) : \phi \in \mathcal{S}(x),~s \in \dom \phi \cap [0,t], \right. \nonumber \\ & ~~~ \left. \not \exists s' \in [0,t] \cap \dom \phi~\mbox{s.t.}~ s' > s  \right\}. \label{eq.Rb}
\end{align}
In words, the set $R^b(t,x)$ includes only the last element reached by each maximal solution to \eqref{eq.1} starting from $x$ over the interval $[0,t]$.

The following lemma can be found in \cite[Theorem 1]{aubin2012differential}.
\begin{lemma} \label{lemAub} 
Suppose that the system in \eqref{eq.1} is forward complete and that $F$ satisfies Assumption \ref{item:difinc}. Then, the following hold for each $t \geq 0$:
\begin{enumerate}
\item The map $x \mapsto R^b(t,x)$ is outer semicontinuous and locally bounded.
\item The map $x \mapsto \mathcal{A}(t,x)$ is outer semicontinuous and locally bounded, where   $\mathcal{A} : \mathbb{R}_{\geq 0} \times \mathbb{R}^n \rightrightarrows \mathcal{S}(\mathbb{R}^n)$ is given by 
\begin{align}
\mathcal{A}(t,x) := & \{ \phi : \phi \in \mathcal{S}(x), ~  
\dom \phi = [0,t] \}. \label{eq.A}
\end{align}
\end{enumerate}
\end{lemma}  

The following lemma recalls the well-known Filippov Theorem that can be found in \cite[Theorem 5.3.1]{Aubin:1991:VT:120830}. 

\begin{lemma} [Filippov Theorem]  \label{lemFil}
Consider the system in \eqref{eq.1} and suppose that $F$ is locally Lipschitz on a compact set $K \subset \mathbb{R}^n$; namely, there exists $\lambda > 0$ such that $ F(y) \subset F(x) + \lambda |x-y| \mathbb{B}$  for all $(x,y) \in K \times K$.  Assume further that $F(x)$ is closed for all $x \in \mathbb{R}^n$. 
Then, for any $(x,y) \in K \times K$ and $t > 0$ such that $\left( R(t,x), R(t, y) \right) \subset K \times K$, each solution $\phi$ to \eqref{eq.1} starting from $x$ satisfies
$| \phi(s,x) |_{R^b(s, y)} \leq \exp (\lambda s) |x - y|$ for all $s \in [0,t]$, where the map $R^b$ is introduced in \eqref{eq.Rb}.
\end{lemma}

\subsection{Proof of Proposition \ref{propout}}

To prove the first item using Lemma \ref{lemAub}, we start showing outer semicontinuity of $R^b$ in \eqref{eq.Rb}. Let $(t_o,x_o) \in \mathbb{R}_{\geq 0} \times \mathbb{R}^n$ and let two sequences $\left\{ (t_{oi},x_{oi}) \right\}^{\infty}_{i=0}$ and $\left\{ y_i \right\}^{\infty}_{i=0}$ be such that $\lim_{i \rightarrow \infty} (t_{oi}, x_{oi}) = (t_o,x_o)$, $y_i \in R^b(t_{oi},x_{oi})$, and $\lim_{i \rightarrow \infty} y_i = y \in \mathbb{R}^n$. Outer semicontinuity of $R^b$ at $(t_o,x_o)$ follows if we show that $y \in R^b(t_o,x_o)$. To this end, we introduce  
$\underline{t} := \min\{t_o,\inf \{t_{oi} : i \in \mathbb{N}\}\}$ and $\bar{t} := \max \{t_o , \sup \{t_{oi} : i \in \mathbb{N} \}\}$. Furthermore, consider a sequence of solutions $\left\{\phi_i \right\}^{\infty}_{i=0}$ to \eqref{eq.1} such that each solution $\phi_i$ starts from $x_{oi}$, $\dom \phi_i = [0,t_{oi} - \underline{t}]$, and $y_i \in R^b(\underline{t}, z_i)$, where $z_i := \phi_i(t_{oi} - \underline{t},x_{oi})$. Now, since the sequence $\left\{(t_{oi},x_{oi}) \right\}^{\infty}_{i=0}$ is uniformly bounded, the solutions to \eqref{eq.1} are forward complete, and since $x \mapsto \mathcal{A}(\bar{t},x)$ is locally bounded, we conclude that the sequence $\{\phi_i\}^{\infty}_{i=0}$ is uniformly bounded. Hence, by passing to an adequate subsequence, we conclude the existence of a function 
$\phi : \dom \phi \rightarrow \mathbb{R}^n$ such that $\lim_{i \rightarrow \infty} \phi_i = \phi$; hence, 
$\phi(0,x_o) = x_o$ and $\dom \phi = [0,t_o-\underline{t}]$. 
The function $\phi$ is a solution to \eqref{eq.1} since the map 
$x \mapsto \mathcal{A}(\bar{t},x)$ is outer semicontinuous via item 2) of Lemma \ref{lemAub}. Furthermore, we note that $\lim_{i \rightarrow \infty} z_i = 
\lim_{i \rightarrow \infty} \phi_i(t_{oi} - \underline{t},x_{oi}) = \phi(t_o - \underline{t},x_o) =: z$. Finally, using the first item in Lemma \ref{lemAub}, we conclude that $y \in R^b (\underline{t},z) \subset R^b (t_o,x_o)$; thus, $R^b$ is outer semicontinuous. 

Now, to show outer semicontinuity of $R$, we consider two sequences $\left\{(t_{oi},x_{oi})\right\}^{\infty}_{i=0}$ and 
$\left\{ y_i \right\}^{\infty}_{i=0}$ such that 
$\lim_{i \rightarrow \infty} (t_{oi}, x_{oi}) = (t_o,x_o)$, 
$y_i \in R(t_{oi},x_{oi})$, and 
$\lim_{i \rightarrow \infty} y_i = y \in \mathbb{R}^n$. 
Outer semicontinuity of $(t,x) \mapsto R(t,x)$ at $(t_o,x_o)$ follows if we show that $y \in R(t_o,x_o)$. Having $y_i \in R(t_{oi},x_{oi})$, for each $i \in \mathbb{N}$, implies the existence of $t'_i \in [0,t_{oi}]$ such that $y_i \in R^b(t'_{i},x_{oi})$, for each $i \in \mathbb{N}$. By passing to an adequate subsequence, we conclude the existence of $t' \in [0,t_o]$ such that $t' = \lim_{i \rightarrow \infty} t'_i$. Hence, since $R^b$ is outer semicontinuous, we conclude that 
$y \in R^b(t',x_o) \subset R(t_o,x_o)$. 

Next, we show that $R^b$ is locally bounded using contradiction. That is, assume the existence of a sequence 
$\{(t_{oi},x_{oi})\}^{\infty}_{i=0}$ such that 
$\lim_{i \rightarrow \infty} (t_{oi},x_{oi}) = (t_o,x_o)$ and
\begin{align} \label{eqcontradict} 
\forall \epsilon > 0, \exists i_\epsilon \in \mathbb{N} 
~\mbox{s.t.}~ R^b(t_{oi},x_{oi}) \not\subset \epsilon \mathbb{B} ~~ \forall i \geq i_\epsilon. 
\end{align}
Note that $R^b(t_{oi},x_{oi}) \subset R^b (\bar{t}, z_i)$, where $z_i := \phi_i(t_{oi} - \bar{t},x_{oi})$ and $\phi_i$ is a backward solution to \eqref{eq.1} starting from $x_{oi}$ with 
$\dom \phi_i = [t_{oi} - \bar{t}, 0]$. Since the backward solutions to \eqref{eq.1} are the forward solutions to \eqref{eq.1back}, the solutions to \eqref{eq.1back} are forward complete, and since $x \mapsto \mathcal{A}(\bar{t},x)$ for the system \eqref{eq.1back} is locally bounded, we conclude that the sequence $\{\phi_i\}^{\infty}_{i=0}$ is uniformly bounded. Hence, by passing to an adequate subsequence, we conclude the existence of a function $\phi : \dom \phi \rightarrow \mathbb{R}^n$ such that $\lim_{i \rightarrow \infty} \phi_i = \phi$; hence, 
$\phi(0,x_o) = x_o$ and $\dom \phi = [t_o-\bar{t},0]$. 
The function $\phi$ is a backward solution to \eqref{eq.1} since the map $x \mapsto \mathcal{A}(t',x)$ for \eqref{eq.1back} is outer semicontinuous. Furthermore, we note that $ \lim_{i \rightarrow \infty} z_i = \lim_{i \rightarrow \infty} \phi_i(t_{oi} - t',x_{oi}) = \phi(t_o - t',x_o) =: z$. The later contradicts \eqref{eqcontradict} since, using Lemma \ref{lemAub}, $x \mapsto R^b(\bar{t},x)$ is locally bounded.

Now, we show that $R$ is locally bounded via contradiction. Assume the existence of a sequence 
$\{(t_{oi},x_{oi})\}^{\infty}_{i=0}$ such that $\lim_{i \rightarrow \infty} (t_{oi},x_{oi}) = (t_o,x_o)$ and,  
\begin{align} \label{eqcontradict1} 
\forall \epsilon > 0, \exists i_\epsilon \in \mathbb{N} 
~\mbox{s.t.}~ R(t_{oi},x_{oi}) \not\subset \epsilon \mathbb{B} \qquad \forall i \geq i_\epsilon. 
\end{align}
This implies the existence of $t'_i \in [0,t_{oi}]$, for all $i \in \mathbb{N}$, such that
\begin{align} \label{eqcontradict2} 
\forall \epsilon > 0, \exists i_\epsilon \in \mathbb{N} 
~\mbox{s.t.}~ R^b(t'_i,x_{oi}) \not\subset \epsilon \mathbb{B} \qquad \forall i \geq i_\epsilon. 
\end{align}
By passing to an adequate subsequence, we conclude the existence of $t' \in [0,t_o]$ such that $t' = \lim_{i \rightarrow \infty} t'_i$. Having $R^b$ locally bounded  contradicts \eqref{eqcontradict2} and, thus, $R$ is locally bounded.  

To prove the second item in Proposition \ref{propout}, given $x \in \mathbb{R}^n$, we establish continuity of the set-valued map $t \mapsto R(t,x)$. Since $t \mapsto R(t,x)$ is outer semicontinuous and locally bounded, it is enough to show that it is lower semicontinuous. We show lower semicontinuity of $t \mapsto R^b(t,x)$ in \eqref{eq.Rb} via contradiction. Assume that there exist $\epsilon>0$, $t_o \geq 0$, $y \in R^b(t_o,x)$, and a sequence $\{t_{oi} \}^{\infty}_{i=0}$ such that $\lim_{i \rightarrow \infty} t_{oi} = t_o$ and, at the same time, 
\begin{align} \label{eqcon}
|y - z| \geq \epsilon \qquad \forall z \in R^b(t_{oi},x), \quad \forall i \in \mathbb{N}.
\end{align}
Consider a maximal solution $\phi$ to \eqref{eq.1} starting from $x$ such that $\phi(t_o,x) = y$ and let $y_i := \phi(t_{oi},x)$. Note that $y_i \in R^b(t_{oi},x)$.  Since the solution $\phi$ is continuous, it follows that $\lim_{i \rightarrow \infty} |y- y_i| =0$, which contradicts \eqref{eqcon}. Now, to show lower semicontinuity of $t \mapsto R(t,x)$, we assume that there exist $\epsilon>0$, $t_o \geq 0$, $y \in R(t_o,x)$, and a sequence $\{t_{oi} \}^{\infty}_{i=0}$ such that $\lim_{i \rightarrow \infty} t_{oi} = t_o$ and, at the same time, 
\begin{align} \label{eqcon1}
|y - z| \geq \epsilon \qquad \forall z \in R(t_{oi},x), \quad \forall i \in \mathbb{N}.
\end{align}
Note that \eqref{eqcon1} implies the existence of $t' \in [0,t_o]$ such that $y \in R^b(t',x)$. Moreover, for each sequence 
$\{t'_i\}^{\infty}_{i=0}$ such that  $t'_i \in [0,t_{oi}]$ and
$\lim_{i \rightarrow \infty} t'_i = t'$, we have 
\begin{align} \label{eqcon2}
|y - z| \geq \epsilon \qquad \forall z \in R^b(t'_{i},x), \quad \forall i \in \mathbb{N}.
\end{align}
However, \eqref{eqcon2} contradicts lower semicontinuity of the map $t \mapsto R^b(t,x)$. 
\hfill $\blacksquare$

\subsection{Proof of Proposition \ref{proplip.1}}

To show that the set-valued map $R$ is locally Lipschitz, we will first show that the map $R^b$ in \eqref{eq.Rb} is locally Lipschitz.  To that end, we consider $(t_o, x_o) \in \mathbb{R}_{\geq 0} \times \mathbb{R}^n$ and the set
\begin{align} \label{eq.pr1} 
U_r(t_o,x_o) := & \left\{ (t, x) \in \mathbb{R}_{\geq 0} \times \mathbb{R}^n : \right. \nonumber \\ & \left. t \in [0, t_o + r], 
~ |x-x_o| \leq r \right\}, 
\end{align}
for some $r>0$. Furthermore, let $\lambda_K > 0$ be the Lipschitz constant of $F$ on the set $K := R(U_r(t_o, x_o))$. 
Note that $K$ is compact since the system is forward complete. 
Next, we show the existence of $\epsilon>0$ such that, for any 
$\left( (t_1,x_1), (t_2, x_2) \right) \in U_r(t_o,x_o) \times U_r(t_o,x_o)$, for any 
$y_1 \in R^b([0,t_1],x_1)$ there exists $y_2 \in R^b(t_2, x_2)$ such that 
\begin{align} \label{eq.pr2}
|y_1 - y_2| \leq & \epsilon \left( |x_1 - x_2| + |t_1 - t_2| \right).
\end{align}
 The latter inequality is enough to conclude that $R^b$ is locally Lipschitz. Let $\left( (t_1,x_1), (t_2, x_2) \right) \in U_r(t_o,x_o) \times U_r(t_o,x_o)$, assume without loss of generality that $t_2 \geq t_1$, and note that both $R(t_2, x_2)$ and $R(t_1, x_1)$ belong to the compact set $K$. Hence, using Lemma \ref{lemFil}, we conclude that $| y_1 |_{R^b(t_1, x_2)} \leq \exp^{\lambda_K (t_o + r)} |x_2 - x_1|$. Thus, for  $y'_2 := \argmin \left\{ y_1 - y : y \in R^b(t_1, x_2) \right\}$, we have $|y_1 - y'_2| \leq \exp^{\lambda_K (t_o + r)} |x_2 - x_1|$. Furthermore, for any $y_2 \in R^b(t_2 - t_1, y'_2)$ and since $F$ is locally bounded, we conclude that
\begin{align*} 
M_r(t_o,x_o) & := \max \{|F(\phi(\tau, y))|: \phi \in \mathcal{S}(y), \\ &  y \in R(U_r(t_o, x_o)),~ \tau \in [0, t_o + r] \} < \infty,
\end{align*} 
where $\mathcal{S}(y)$ is the set of maximal solutions to system \eqref{eq.1} starting from $y$. Hence,
$|y_2 - y'_2| \leq M_r(t_o,x_o) |t_1 - t_2|$
and $|y_1 - y_2| \leq | y_1 - y'_2 | + | y'_2 - y_2 |  
\leq \epsilon  \left( |x_2 - x_1| + |t_1 - t_2| \right)$,
where 
\begin{align} \label{eq.esp} 
\epsilon := & \max \left\{ \exp^{\lambda_K (t_o + r)}, M_r(t_o,x_o)  \right\}.
\end{align} 

Now, to show that the set-valued map $R$ is locally Lipschitz, we consider $(t_o, x_o) \in \mathbb{R}_{\geq 0} \times \mathbb{R}^n$ and the compact neighborhood 
$U_r(t_o,x_o)$ introduced in \eqref{eq.pr1}. We will show the existence of $\epsilon > 0$ such that for any two elements $\left( (t_1,x_1), (t_2, x_2) \right) \in U_r(t_o,x_o) \times U_r(t_o, x_o)$, for any $y_1 \in R(t_1, x_1)$ we can find $y_2 \in R(t_2, x_2)$ such that \eqref{eq.pr2} holds. Indeed, consider $t_{y_1} \in [0,t_1]$ such that $ R^b(t_{y_1}, x_1) = y_1$ and 
$(t_{y_1}, x_1) \in U_r(t_o,x_o)$. Hence, 
there exists $y_2 \in R^b(t_{y_2}, x_2)$ with $t_{y_2}$ the closest element to $t_{y_1}$ while being in $[0,t_2]$. Note that $(t_{y_2}, x_2) \in U_r(t_o,x_o)$ and since $R^b$ is locally Lipschitz, we have 
$|y_1 - y_2| \leq  \epsilon \left( |t_{y_1}-t_{y_2}| + 
|x_1-x_2| \right) \leq  \epsilon \left( |t_{1}-t_{2}| + 
|x_1-x_2| \right)$,
where $\epsilon$ is introduced in \eqref{eq.esp}. 
\hfill $\blacksquare$

\subsection{Proof of Lemma \ref{lem3}}

We prove item 1 by directly showing that $B$ satisfies the definition of lower semicontinuity for scalar functions. That is, for every  sequence $\left\{ z_i \right\}_{i=0}^{\infty} \subset \mathbb{R}^m$ such that 
$\lim_{i \rightarrow \infty} z_i = z_o$, we show that 
$\liminf_{i \rightarrow \infty} f(z_i) = 
\liminf_{i \rightarrow \infty} \min_{y \in \Pi(z_i)} |y|_X \geq \min_{y \in \Pi(z_o)} |y|_X = f(z_o)$ 
provided that the set-valued map $\Pi$ is outer semicontinuous in which case, since 
$\Pi$ is already locally bounded, $\inf$ in $f$ becomes $\min$. Since the map $\Pi$ is outer semicontinuous, we conclude that, for all $y_i \in \Pi(z_i)$ such that $\lim_{i \rightarrow \infty} y_i = y_o \in \mathbb{R}^n$, we have $y_o \in \Pi(z_o)$.  Choose $\{y_i\}^{\infty}_{i=0}$ to be such that $y_i \in \Pi(z_i)$ and $|y_i|_X = \min_{y \in \Pi(z_i)} |y|_X$ for each $i \in \mathbb{N}$. Hence, 
$\liminf_{i \rightarrow \infty} f(z_i) = \liminf_{i \rightarrow \infty}  \min_{y \in \Pi(z_i)} |y|_X = \liminf_{i \rightarrow \infty} |y_i|_X$. 
Since the distance function to $X$ is continuous, we conclude that 
$\liminf_{i \rightarrow \infty} f(z_i) = \liminf_{i \rightarrow \infty}  \min_{y \in \Pi(z_i)} |y|_X = 
|\liminf_{i \rightarrow \infty} y_i|_X$. Since $\Pi$ is locally bounded, the sequence $\left\{y_i\right\}^{\infty}_{i=0}$ is bounded; hence, $\liminf_{i \rightarrow \infty} y_i = y_o \in \mathbb{R}^n$. Moreover, by passing to a suitable sub-sequence $ \left\{y_{i_k}\right\}^{\infty}_{k=0} $, we conclude that $ \liminf_{i \rightarrow \infty} y_i = \lim_{k \rightarrow \infty} y_{i_k}= y_o $. Thus, since $\Pi$ is outer semicontinuous, it follows that $y_o \in \Pi(z_o)$. Finally, $\liminf_{i \rightarrow \infty} f(z_i) = |y_o|_X \geq \min_{y \in \Pi(z_o)} |y|_X = f(z_o) $.
We prove item 2 by directly using the definition of upper semicontinuity for scalar functions. That is, we show that, for every sequence $\left\{ z_i \right\}_{i=0}^{\infty} \subset \mathbb{R}^m$ such that 
$\lim_{i \rightarrow \infty} z_i = z_o$, 
we have $\limsup_{i \rightarrow \infty} f(z_i) = 
\limsup_{i \rightarrow \infty} \min_{y \in \Pi(z_i)} |y|_X \leq  \min_{y \in \Pi(z_o)} |y|_X = f(z_o)$
provided that the set-valued map $\Pi$ is lower semicontinuous. To reach a contradiction, we assume the existence of a sequence 
$\left\{ z_i \right\}^{\infty}_{i=0}$ such that 
$\lim_{i \rightarrow \infty} z_i = z_o$ and
$\lim_{i \rightarrow \infty} f(z_i) > f(z_o)$.
The latter implies the existence of $\epsilon > 0$ and $i_o \in \mathbb{N}$ such that, for all $i \geq i_o$, 
\begin{align} \label{eqlem1} 
 f(z_i) - f(z_o) =  \inf_{y \in \Pi(z_i)} |y|_X - \inf_{y \in \Pi(z_o)} |y|_X > \epsilon. 
 \end{align}
Let $z_o := \mbox{arg inf}_{y \in \Pi(z_o)} |y|_{X}$, 
and 
\begin{align}\label{eqlem3} 
w_i := \mbox{arg inf}_{y \in \Pi(z_i)} |y|_{X} ~~~~ \forall i \in \mathbb{N}. 
\end{align}
Using \eqref{eqlem1}, we conclude that $|w_i|_X - |w_o|_X > \epsilon$. On the other hand, since the set-valued map $\Pi$ is lower semicontinuous, it follows that there exists $i_1 \in \mathbb{N}$ such that, for all $i \geq i_1$, there exists  $w'_i \in \Pi(z_i)$ such that $|w'_i - w_o| \leq \epsilon/2$. Using \eqref{eqlem3}, we conclude that, for all $i \geq \max \left\{ i_o, i_1 \right\}$, $|z'_i|_X \geq |z_i|_X $ and $|z'_i|_X - |z_o|_X  \geq |z_i|_X - |z_o|_X > \epsilon$. Finally, since the distance function with respect to the set 
$X$ is globally Lipschitz, we obtain, for all 
$i \geq \max \left\{ i_o, i_1 \right\}$, $\epsilon/2 \geq |w'_i - w_o| \geq |w'_i|_X - |w_o|_X > \epsilon$,
which yields to a contradiction.

To prove the third item, we consider two elements $(z,y) \in \mathbb{R}^n \times \mathbb{R}^n$ and the corresponding two elements $(z',y') \in X \times X$ such that 
\begin{align} 
|z'|_{\Pi(z)}  = & \inf_{w \in X} |w|_{\Pi(z)}  = 
\inf_{w \in \Pi(z)} |w|_{X} = f(z), \label{eqdista} \\ 
|y'|_{\Pi(y)}  =  & \inf_{w \in X} |w|_{\Pi(y)}  = \inf_{w \in \Pi(y)} |w|_{X} = f(y). \label{eqdistb} 
\end{align}
Using the triangular inequality, we conclude that
$|y'|_{\Pi(z)} \leq |y'|_{\Pi(y)} + d_{H}(\Pi(z), \Pi(y))$ and 
$|z'|_{\Pi(y)} \leq |z'|_{\Pi(z)} + d_H(\Pi(z), \Pi(y))$, where $d_H(\Pi(z), \Pi(y))$ is the Hausdorff distance between the two sets $R(z)$ and $R(y)$ introduced in \eqref{eq.Haus}. Furthermore, using the first equality in \eqref{eqdista} and \eqref{eqdistb}, respectively, we conclude that    
 \begin{align} 
|z'|_{\Pi(z)} \leq |y'|_{\Pi(z)}  \leq  &  |y'|_{\Pi(y)} + d_{H}(\Pi(z), \Pi(y)), \label{eqdist2a} \\ 
|y'|_{\Pi(y)} \leq |z'|_{\Pi(y)}  \leq  & |z'|_{\Pi(z)} + d_H(\Pi(z),\Pi(y)).  \label{eqdist2b}
\end{align}
Hence, using \eqref{eqdist2a}-\eqref{eqdist2b} and the second equality in \eqref{eqdista} and \eqref{eqdistb}, respectively, we obtain $|f(z) - f(y)| \leq  |\Pi(z) - \Pi(y)|$. Finally, when the map $\Pi$ is locally Lipschitz, using Definition \ref{deflip}, we conclude the existence of $\lambda > 0$ such that $|f(z) - f(y)| \leq  |\Pi(z) - \Pi(y)| \leq \lambda |z -y|$. 
\hfill $\blacksquare$

\subsection{Proof of Lemma \ref{lem2}}

Given a compact set $\mathcal{I} \subset \mathbb{R}^n$ such that 
$\mathcal{I} \cap K = \emptyset$ and the continuous function $h$,
we introduce the sequence $\left\{ \eta_k \right\}^{\infty}_{k=1}$ given by
\begin{align} \label{eq.seq}
\eta_k := \min \{ h(t,x) : x \in \mathcal{I}, ~ t \in [0, k] \}.
\end{align}
This sequence is strictly positive and nonincreasing. 

Next, we propose to partition the set 
$\mathbb{R}_{\geq 0}$ 
using an increasing sequence 
$\{t_i\}^{\infty}_{i = 0} \subset \mathbb{R}_{\geq 0}$ 
that we design as follows:

\begin{enumerate}
\item For each interval $T_k := [k-1, k]$, $k \in \mathbb{N}^*$, we associate $u_k \in \mathbb{N}^*$. Furthermore, we introduce the sequence $\{ j_k \}^{\infty}_{k=1}$ such that $j_1 := 0$ and $j_{k+1} := j_k + u_k$. 

\item The subsequence $\left\{ t_i \right\}^{u_1}_{i= 0}$ satisfies $t_0 := 0$ and $t_{i+1} := t_i + \frac{1}{u_{1}}$ for all $i \in \{0,1,...,  u_1-1 \}$. It follows that $t_{u_1} = 1$.

\item For each $k \geq 2$, the subsequence 
$\left\{ t_i \right\}^{ j_k + u_k}_{i= j_k}$ satisfies $t_{j_k} = k-1$ and $t_{i+1} := t_i + \frac{1}{u_{k}}$ for all $i \in \{j_k,j_k + 1,...,  j_k+u_k-1 \}$. It follows that $t_{j_k + u_k} = t_{j_{k+1}} = k$.

\item Under the continuity of $h$, we choose the parameter $u_{k}$ such that, for each $i \in \{j_k,j_k + 1,..., j_k + u_k - 1 \}$ and for each $x \in \mathcal{I}$,
$ h(t_i,x) - h(t_{i+1},x) = h(t_i,x) - h(t_{i} + (1/u_{k}), x) < \frac{1}{4} \eta_{k}$.
\end{enumerate}

Now, we consider a nonincreasing sequence $\left\{ \zeta_i \right\}^{\infty}_{i=0} \subset \mathbb{R}_{>0}$ such that
\begin{align} \label{eq.seq3}
 \sum^{\infty}_{i = j_k} \zeta_{i} < \frac{1}{8} \eta_k.
\end{align}
Furthermore, using the continuity of $h$, we conclude the existence of a sequence of functions $\left\{w_i\right\}^{\infty}_{i=0}$
such that: Each $w_i : \mathbb{R}^n \rightarrow \mathbb{R}_{>0}$is continuously differentiable on $\mbox{int}(\mathcal{I})$. For each $x \in \mathcal{I}$, the sequence $\left\{w_i(x)\right\}^{\infty}_{i=0}$ is nonincreasing.  For each $ i \in \mathbb{N}$,
\begin{align} \label{eq.seq4}
| h(t_i, x) - w_i(x) | < \frac{1}{2} \zeta_i  + \sum^{\infty}_{l = i} \zeta_{l}.
\end{align}
Finally, we construct the function $t \mapsto g(t,x)$ by interpolating the sequence of functions $\{w_i(x)\}^{\infty}_{i=0}$ by means of a nonincreasing third order polynomial to obtain, for any $t \in [t_i, t_{i+1}]$ and $i \in \mathbb{N}$,  $g(t,x) := q ( t, t_i, t_{i+1}, w_i(x), w_{i+1}(x))$,  where 
\begin{align*} 
 q(t, t_i, t_{i+1}, w_i(x), w_{i+1}(x)) & := w_i(x) + \\ &  (w_{i+1}(x) - w_i(x))  \frac{3 (t-t_i)^2}{(t_{i+1} - t_i)^2} - \\ &  (w_{i+1}(x) - w_i(x))  \frac{2 (t-t_i)^3}{(t_{i+1} - t_i)^3}.
\end{align*}
Note that $q$ is nonincreasing on $[t_i,t_{i+1}]$ and
\begin{align*} 
q ( t_i, t_i, t_{i+1}, w_i(x), w_{i+1}(x) ) & = w_i(x), 
\\
 q ( t_{i+1}, t_i, t_{i+1}, w_i(x), w_{i+1}(x) ) & = w_{i+1}(x),   
 \\
 \dot{q} ( t_i, t_i, t_{i+1}, w_i(x), w_{i+1}(x) ) & = 0, 
 \\
 \dot{q} ( t_{i+1}, t_i, t_{i+1}, w_i(x), w_{i+1}(x) ) & = 0. 
 \end{align*}
In order to complete the proof, it remains to show that \eqref{eq.eqnc3} is satisfied for all $(t,x) \in \mathbb{R}_{\geq 0} \times \mathcal{I}$. Without loss of generality, consider $x \in \mathcal{I}$ and $t \in [k-1, k)$, for $k \in \{1,2,...,\infty \}$. Assume that 
$t \in [t_i, t_{i+1}]$ for some $t_i \in [k-1, k)$. Hence, 
$i \in \left( j_k, j_k + u_k \right)$. It follows that
$ g(t,x) - h(t,x)  \leq  g(t_i,x) - h(t_{i+1},x) 
\leq  | g(t_i,x) - h(t_{i},x) | + h(t_{i},x) - h(t_{i+1},x) 
\leq  \sum^{\infty}_{j=i} \zeta_j + \frac{1}{2} \zeta_i + \frac{1}{4} \eta_{k} < \frac{1}{2} \eta_{k}$,
where we used the fact that $ g(t_i,x) = w_i(x) $, \eqref{eq.seq4}, and \eqref{eq.seq3}. Similarly,
\begin{align*}
 h(t,x) - & g(t,x) \leq h(t_{i+1},x) - g(t_i,x) \\ &
\leq  | g(t_{i+1},x) - h(t_{i+1},x) | + h(t_{i},x) - h(t_{i+1},x) 
\\ &
\leq  \sum^{\infty}_{j=i+1} \zeta_j + \frac{1}{2} \zeta_i + \frac{1}{4} \eta_{k} < \frac{1}{2} \eta_{k}.
\end{align*}
Therefore, $|h(t,x) - g(t,x)| \leq  \frac{1}{2} \eta_{k}$ and $h(t,x) - \frac{1}{2} \eta_{k} \leq  g(t,x) \leq h(t,x) + \frac{1}{2} \eta_{k}$.
Finally, using \eqref{eq.seq}, we conclude that
$ \eta_{k} \leq \min \{ h(\tau,x) : \tau \in [0, k] \}$
and, since $ t \in [k-1, k)$, it follows that $\eta_{k} \leq h(t,x)$. 
\hfill $\blacksquare$

\subsection{Proof of Lemma \ref{lem1smooth}}

We propose to adapt the proof 
of \cite[Lemma 48.3]{hahn1967stability} to the case where the origin is replaced by a general closed set $K \subset \mathbb{R}^n$.  For each integer $s$, we introduce the set
\begin{align} \label{eq.set}
I_s:= \left\{ 
x \in \mathbb{R}^n : 2^{s-3} \leq |x|^2_K \leq 2^{s+4}   
\right\}.
\end{align} 
Furthermore, we propose to decompose the set $I_s$ into a sequence of nonempty compact subsets $\{D^{s}_i\}^{N}_{i=1}$, where $N \in \{1,2,...,\infty\}$, such that $D^s_i \subset I_s$ for all $i \in \{1,2,...,N\}$. 
Furthermore, for each $i \in \{1,2,...,N\}$, there exist a finite set $\mathcal{N}^s_i \subset \{1,2,...,N \}$ and a compact set $\bar{D}^s_i$ including $D^s_i$ in its interior such that $\bar{D}^s_i \cap K = \emptyset$, $D^s_i \subset \bigcup_{j \in \mathcal{N}^s_i} D^s_j$, and $\bar{D}^s_i \cap \bar{D}^s_j = \emptyset$ for all $j \notin \mathcal{N}^s_i$.   

The rest of the proof follows in three steps. 
\begin{enumerate}
\item In the first step, we use Lemma \ref{lem2} to construct a function $\psi^s_{i} : \mathbb{R}_{\geq 0} \times \bar{D}^s_i \rightarrow \mathbb{R}_{\geq 0}$ that is nonincreasing with respect to its first argument, $\mathcal{C}^1$ on $\mathbb{R}_{\geq 0} \times \mbox{int}(\bar{D}^s_i)$, and satisfies \eqref{eq.eqnc} for all 
$(t,x) \in \mathbb{R}_{\geq 0} \times \bar{D}^s_i$.

\item In the next step, we consider an open set 
$O^s_i \subset \bar{D}^s_i$ that contains $D^s_i$, and a differentiable function $\lambda^s_{i} : \mathbb{R}^n \rightarrow \mathbb{R}_{\geq 0}$, which is positive in $O^s_i$ and vanishes outside. Then, we introduce the function $\psi_s (t,x) := \frac{1}{\lambda_s(x)} 
\sum^{N}_{i=1} \psi^s_i(t,x) \lambda^s_i(x)$ with 
$\lambda_s(x) := \sum^N_{i=1} \lambda^s_{i}(x)$. 
Note that, for each $x \in I_s$, the previous sum is finite by construction of the sequence $\{D^{s}_i\}^{N}_{i = 1}$. Furthermore,  the map $t \mapsto \psi_s(t,x)$ nonincreasing , $\psi_{s} \in \mathcal{C}^1(\mathbb{R}_{\geq 0} \times \mbox{int}(I_s))$, and satisfies \eqref{eq.eqnc} for all 
$(t,x) \in \mathbb{R}_{\geq 0} \times I_s$. 

\item In the last step, we consider $g(t,x) := \frac{1}{\lambda(x)} \sum^{+ \infty}_{s = - \infty}  \psi_{s}(t,x) \lambda_{s}(x)$, $\lambda(x) := \sum^{+\infty}_{s =-\infty} \lambda_{s}(x)$.
Finally, it is easy to see that for all $x \in \mathbb{R}^{n}$, the previous sum is finite. 
\end{enumerate}
\hfill $\blacksquare$

\balance

\bibliographystyle{ieeetr}       
\bibliography{biblio}

\section*{Biography}

\begin{wrapfigure}{l}{25mm} 
\includegraphics[width=1in,height=1.25in,clip,
keepaspectratio]{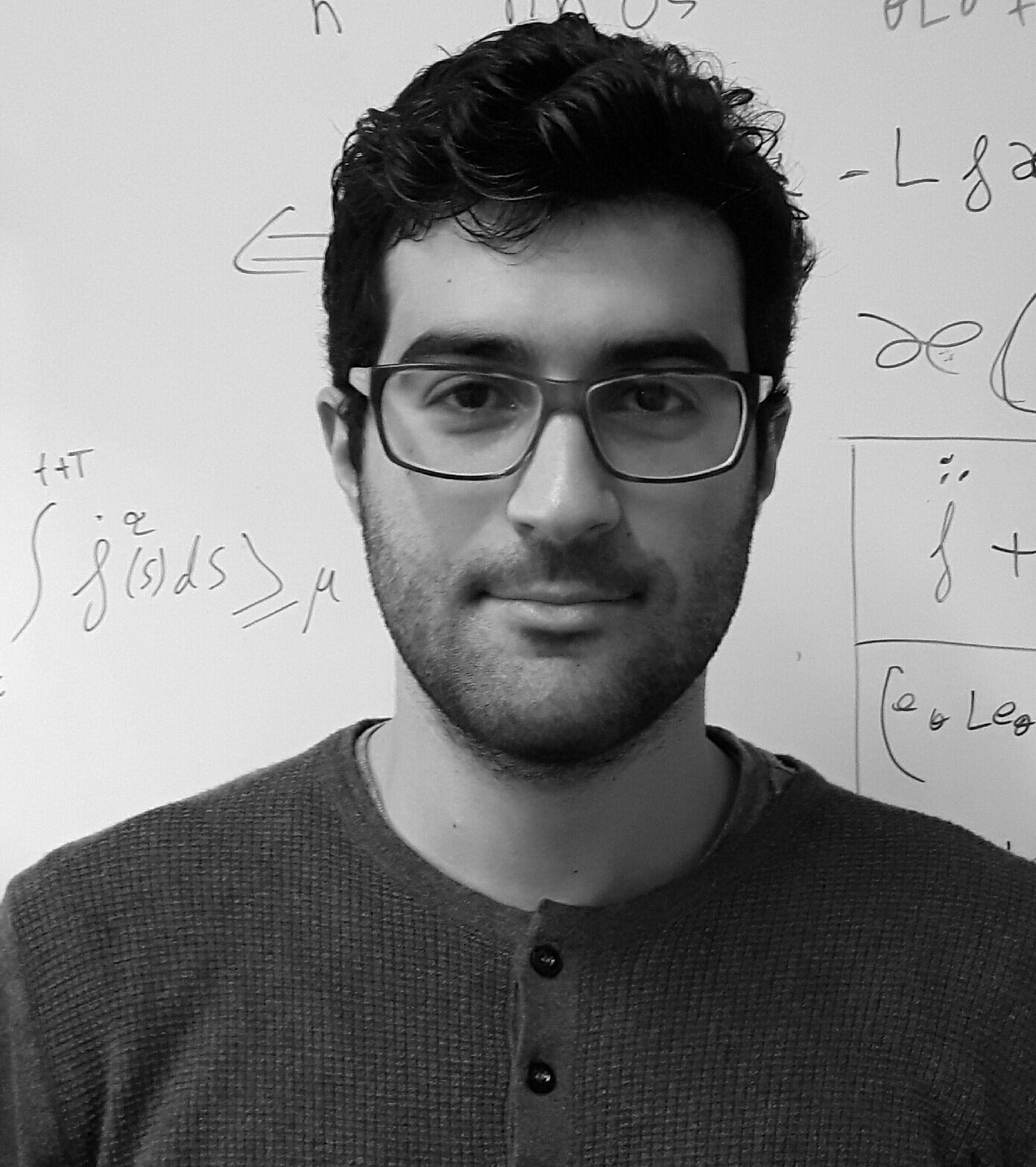}
\end{wrapfigure} 
\par \textbf{Mohamed Maghenem} received his Control-Engineer degree from the Polytechnical School of Algiers, Algeria, in 2013, his M.S. and Ph.D. degrees in Automatic Control from the University of Paris-Saclay, France, in 2014 and 2017, respectively. He was a Postdoctoral Fellow at the Electrical and Computer Engineering Department at the University of California at Santa Cruz from 2018 through 2021.
 M. Maghenem has the honour of holding a research position at the French National Centre of Scientific Research (CNRS) since January 2021.  His research interests include dynamical systems theory (stability, safety, reachability, robustness, and synchronization), control systems theory (adaptive, time-varying, linear, non-linear, hybrid, robust, etc.) with applications to power systems, mechanical systems, and cyber-physical systems.
\par

\begin{wrapfigure}{l}{25mm} 
\includegraphics[width=1in,height=1.25in,clip,keepaspectratio]{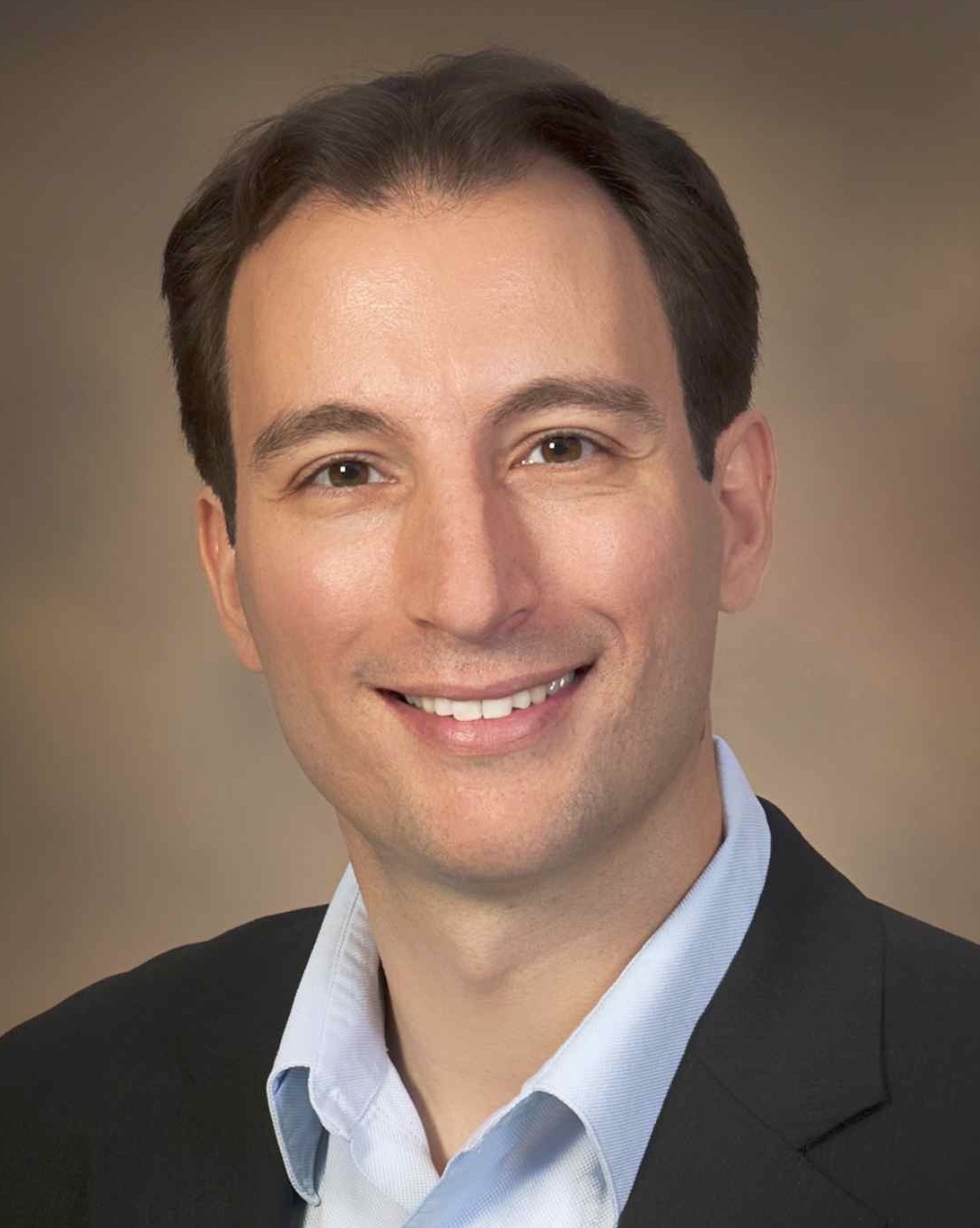}
\end{wrapfigure} 
\par \textbf{Ricardo. G. Sanfelice} received the B.S. degree in Electronics Engineering from the Universidad de Mar del Plata, Buenos Aires, Argentina, in 2001, and the M.S. and Ph.D. degrees in Electrical and Computer Engineering from the University of California, Santa Barbara, CA, USA, in 2004 and 2007, respectively. In 2007 and 2008, he held postdoctoral positions at the Laboratory for Information and Decision Systems at the Massachusetts Institute of Technology and at the Centre Automatique et Systèmes at the École de Mines de Paris. In 2009, he joined the faculty of the Department of Aerospace and Mechanical Engineering at the University of Arizona, Tucson, AZ, USA, where he was an Assistant Professor. In 2014, he joined the University of California, Santa Cruz, CA, USA, where he is currently Professor in the Department of Electrical and Computer Engineering. Prof. Sanfelice is the recipient of the 2013 SIAM Control and Systems Theory Prize, the National Science Foundation CAREER award, the Air Force Young Investigator Research Award, the 2010 IEEE Control Systems Magazine Outstanding Paper Award, and the 2020 Test-of-Time Award from the Hybrid Systems: Computation and Control Conference. His research interests are in modeling, stability, robust control, observer design, and simulation of nonlinear and hybrid systems with applications to power systems, aerospace, and biology.
\par

\end{document}